\documentclass[12pt]{amsart}

\usepackage{mathptmx}       % selects Times Roman as basic font
\usepackage{helvet}         % selects Helvetica as sans-serif font
\usepackage{courier}        % selects Courier as typewriter font
\usepackage{type1cm}        % activate if the above 3 fonts are
\usepackage{multirow}
\usepackage{booktabs}
\usepackage{graphicx}        % standard LaTeX graphics tool
\usepackage{multicol}        % used for the two-column index
\usepackage{yhmath}        % for an overarc

\usepackage{amstext,amsfonts,amssymb,amscd,amsbsy,amsmath,colordvi}

\usepackage{tabularx}
\newcolumntype{K}[1]{>{\centering\arraybackslash}p{#1}}
\usepackage{float}
\usepackage{pgf}
\usepackage{mathrsfs}
\newtheorem{theorem}{Theorem}     
\numberwithin{theorem}{section}
\newtheorem{lemma}[theorem]{Lemma}     
\newtheorem{corollary}[theorem]{Corollary}     
\newtheorem{proposition}[theorem]{Proposition}     

\theoremstyle{definition}  
     
\newtheorem{example}[theorem]{Example}     
\newtheorem{remark}[theorem]{Remark}     

\newcommand{\CC}{\mathbb{C}}

\newcommand{\PP}{\mathbb{P}}

\newcommand{\RR}{\mathbb{R}}
\newcommand{\ZZ}{\mathbb{Z}}

\newcommand{\cO}{\mathcal{O}}

\DeclareMathOperator{\conv}{conv}
\DeclareMathOperator{\PSD}{PSD}
\newcommand{\cC}{\mathcal{C}}
\newcommand{\cE}{\mathcal{E}}
\newcommand{\arc}[1]{\wideparen{#1}}
\newcommand{\defcolor}[1]{{\color{blue} #1}}
\newcommand{\demph}[1]{{\it\defcolor{#1}}}

\title{Convex Hull of Two Circles in $\RR^3$}
\author[E.~Nash]{Evan D.~Nash}     
\address{Evan D.~Nash\\     
         Department of Mathematics\\     
         The Ohio State University\\     
         Columbus\\     
         Ohio \ 43210\\     
         USA}     
\email{nash.228@osu.edu}     
\urladdr{https://math.osu.edu/people/nash.228}
\author[A.~Pir]{Ata Firat Pir}     
\address{Ata Firat Pir\\     
         Department of Mathematics\\     
         Texas A\&M University\\     
         College Station\\     
         Texas \ 77843\\     
         USA}     
\email{atafirat@math.tamu.edu}     
\urladdr{http://www.math.tamu.edu/~atafirat/}
\author[F.~Sottile]{Frank Sottile}     
\address{Frank Sottile\\     
         Department of Mathematics\\     
         Texas A\&M University\\     
         College Station\\     
         Texas \ 77843\\     
         USA}     
\email{sottile@math.tamu.edu}     
\urladdr{http://www.math.tamu.edu/~sottile}
\author[L.~Ying]{Li Ying}     
\address{Li Ying\\     
         Department of Mathematics\\     
         Texas A\&M University\\     
         College Station\\     
         Texas \ 77843\\     
         USA}     
\email{98yingli@math.tamu.edu}     
\urladdr{http://www.math.tamu.edu/~98yingli/}
\thanks{Research of Pir, Sottile, and Li supported in part by NSF grant DMS-1501370.}
\thanks{This article was initiated during the Apprenticeship
Weeks (22 August--2 September 2016), led by Bernd
Sturmfels, as part of the Combinatorial Algebraic
Geometry Semester at the Fields Institute.}

\begin{document}

\begin{abstract}
We describe convex hulls of the simplest compact space curves,
 reducible quartics consisting of two circles. 
 When the circles do not meet in complex projective space, their algebraic boundary contains an 
 irrational ruled surface of degree eight whose ruling forms a genus one curve. 
 We classify which curves arise, classify the face lattices of the convex hulls  
 and determine which are spectrahedra.
 We also discuss an approach to these convex hulls using projective duality.
\end{abstract}

\maketitle
%%%%%%%%%%%%%%%%%%%%%%%%%%%%%%%%%%%%%%%%%%%%%%%%%%%%%%%%%%%%%%%%%%%%%%%%%%%%%%%%%

%%%%%%%%%%%%%%%%%%%%%%%%%%%%%%%%%%%%%%%%%%%%%%%%%%%%%%%%%%%%%%%%%%%%%%%%%%%%%%%%%
\section{Introduction}
Convex algebraic geometry studies convex hulls of semialgebra\-ic sets~\cite{orbitopes}. 
The convex hull of finitely many points, a zero-dimensional variety, is a polytope~\cite{Gr03,Z95}.
Polytopes have finitely many faces, which are themselves polytopes.
The boundary of the convex hull of a higher-dimensional algebraic set typically has infinitely many faces
which lie in algebraic families.
Ranestad and Sturmfels~\cite{RS11} described this boundary using projective duality and
secant varieties.
For a general space curve, the boundary consists of finitely many
two-dimensional faces supported on tritangent planes and a scroll of line segments, called the edge surface.
These segments are stationary bisecants, which join two points of the curve whose 
tangents meet.

We study convex hulls of the simplest nontrivial compact space curves, those
which are the union of two circles lying in distinct planes.
Zero-dimensional faces of such a convex hull are extreme points on the circles.
One-dimensional faces are stationary bisecants. 
It may have two-dimensional faces coming from the planes of the circles.
It may have finitely many nonexposed faces, either points of one circle whose tangent meets the
other circle, or certain tangent stationary bisecants.
Fig.~\ref{Fig:firstTaste} shows some of this diversity.
\begin{figure}[htb]
  \centering
  \raisebox{-37pt}{\includegraphics[height=78pt]{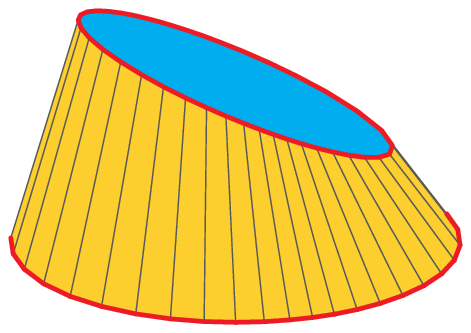}}\quad
  \raisebox{-37pt}{\includegraphics[height=78pt]{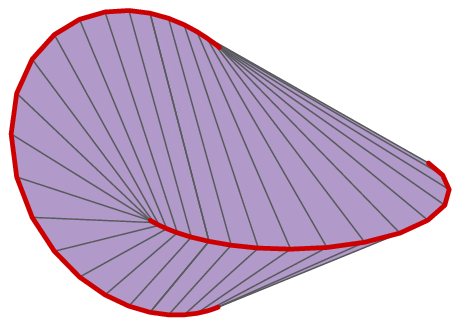}}\quad
  \raisebox{-45pt}{\includegraphics[height=94pt]{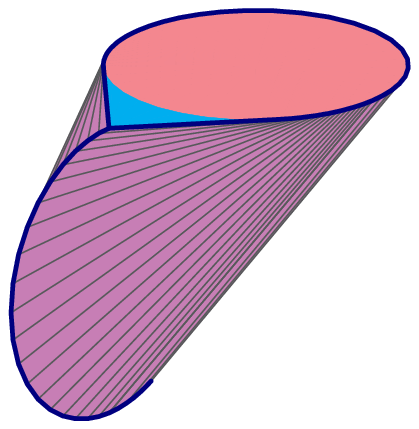}}
  \caption{Some convex hulls of two circles.}\label{Fig:firstTaste}
 \end{figure}
In the convex hull on the left, the discs of both circles are faces, and every face is exposed.
In the oloid in the middle, the discs lie in the interior, an arc %$2/3$ 
of each circle is extreme, and 
the endpoints of the arcs are nonexposed.
In the convex hull on the right, there are two nonexposed stationary bisecants lying on its two-dimensional
face, which is the convex hull of one circle and the point where the other circle is tangent to the 
plane of the first. 

These objects have been studied before.  
Paul Schatz discovered and patented the oloid in 1929~\cite{Schatz}, this is the convex hull of two congruent
circles in orthogonal planes, each passing through the center of the other.
It has found industrial uses~\cite{Oloid_AG}, and is a well-known toy.
A curve in $\RR^3$ may roll along its edge surface.
When rolling, the oloid develops its entire surface and has area equal to that of the 
sphere~\cite{Oloid_developable} with equator one of the circles of the oloid.
Other special cases of the convex hull of two circles have been studied from these perspectives~\cite{Finch,TCR}.

This paper had its origins in Subsection 4.1 of~\cite{RS11}, which claimed that the edge surface for a general pair
of circles is composed of cylinders.
We show that this is only the case when the two circles either meet in two points or are
mutually tangent---in all other cases, the edge surface has higher degree and it is an irrational surface of degree
eight when the circles are disjoint in $\CC\PP^3$.
This is related to Problem~3 on Convexity in~\cite{Sturmfels}, on the
convex hull of three ellipsoids in $\RR^3$.
An algorithm was presented in~\cite{ellipsoids} (see the video~\cite{ellipsoids_video}), using projective
duality.
We sketch this in Section~\ref{S:Duality}, and also apply duality to
the convex hull of two circles. 

In Section~\ref{S:background}, we recall some aspects of convexity and convex algebraic geometry, and 
show that the convex hull of two circles is the projection of a spectrahedron.
We study the edge surface and the edge curve of stationary bisecants of complex conics $C_1,C_2\subset\CC\PP^3$ in 
Section~\ref{S:relaxation}.
We show that the edge curve is a reduced curve of bidegree $(2,2)$ in $C_1\times C_2$ and, if 
$C_1\cap C_2=\emptyset$ and neither circle is tangent to the plane of the other,
then the edge surface has degree eight.
We also classify which curves of bidegree $(2,2)$ arise as edge curves to two conics.
All possibilities occur, except a rational curve with a cusp singularity and a maximally reducible curve.

In Section~\ref{S:circles}, we classify the possible arrangements of two circles lying in
different planes in terms that are relevant for their convex hulls.
We determine the face lattice and the real edge curve of each type,
and show that these convex hulls are spectrahedra only
when the circles lie on a quadratic cone.

%%%%%%%%%%%%%%%%%%%%%%%%%%%%%%%%%%%%%%%%%%%%%%%%%%%%%%%%%%%%%%%%%%%%%%%%%%%%%%%%%
\section{Convex algebraic geometry}\label{S:background}
We review some fundamental aspects of convexity and convex algebraic geometry,
summarize our results about convex hulls of pairs of circles and their edge curves, and show that any such convex
hull is the projection of a  spectrahedron.

%%%%%%%%%%%%%%%%%%%%%%%%%%%%%%%%%%%%%%%%%%%%%%%%%%%%%%%%%%%%%%%%%%%%%%%%%%%%%%%%%
The convex hull of a subset $S\subset\RR^d$ is 
\[
   \defcolor{\conv(S)}\ :=\ 
   \Bigl\{ \sum_{i=1}^n \lambda_i s_i \mid  s_1,\dotsc,s_n\in S\,,\ 
                    0\leq\lambda_i\,,\ \mbox{and } 1=\sum_{i=1}^n \lambda_i\Bigr\}\,.
\]
A set $K$ is \demph{convex} if it equals its convex hull.
A point $p\in K$ is \demph{extreme} if $K\neq\conv(K\smallsetminus\{p\})$.
A compact convex set is the convex hull of its extreme points.

A convex subset $F$ of a convex set $K$ is a \demph{face} if $F$ contains the endpoints of any line segment
in $K$ whose interior meets $F$.
A \demph{supporting hyperplane} $\Pi$ is one that meets $K$ with $K$ lying in one of the half-spaces of $\RR^d$
defined by $\Pi$.
A supporting hyperplane $\Pi$ supports a face $F$ of $K$ if $F\subset K\cap\Pi$, and it 
\demph{exposes} $F$ if $F=K\cap\Pi$.

Not all faces of a convex set are exposed. 
The boundary of the convex hull of two coplanar circles in Fig.~\ref{F:CirclesInPlane}
consists of one arc on each circle and two bitangent segments.
An endpoint $p$ of an arc is not exposed.
The only line supporting  $p$ is the tangent to the circle at $p$, and this line also supports the
adjoining bitangent.
%%%%%%%%%%%%%%%%%%%%%%%%%%%%%%%%%%%%%%%%%%%%%%%%%%%%%%%%%%%%%%%%%%%%%%%%%%%%%%%%%
\begin{figure}[htb]
\centering
 \begin{picture}(120,60)
   \put(0,0){\includegraphics{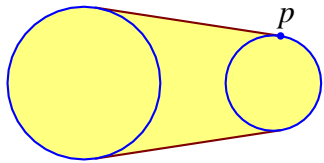}}
   \put(104,54){$p$}
 \end{picture}
 \caption{Convex hull of coplanar circles.}\label{F:CirclesInPlane}
\end{figure}
%%%%%%%%%%%%%%%%%%%%%%%%%%%%%%%%%%%%%%%%%%%%%%%%%%%%%%%%%%%%%%%%%%%%%%%%%%%%%%%%%

A fundamental problem from convex optimization is to describe the faces of a convex set, determining which are
exposed, as well as their lattice of inclusions (the \demph{face lattice}).
For more on convex geometry, see~\cite{barvinok}.

%%%%%%%%%%%%%%%%%%%%%%%%%%%%%%%%%%%%%%%%%%%%%%%%%%%%%%%%%%%%%%%%%%%%%%%%%%%%%%%%%%%%%%%%%%%%%%%%%%%%%%%%
Convex algebraic geometry is the marriage of classical convexity with real algebraic geometry.
A real algebraic variety $X$ is an algebraic variety defined over $\RR$.
If $X$ is irreducible and contains a smooth real point, then its real points are Zariski-dense in $X$, so it is
often no loss to consider only the real points.
Conversely, many aspects of a real algebraic variety are best understood in terms of its complex points.
Studying the complex algebraic geometry aspects of a question from real algebraic geometry is its 
\demph{algebraic relaxation}.
This relaxation enables the use of powerful techniques from complex algebraic geometry to 
address the original question.

As the real numbers are ordered, we also consider \demph{semialgebraic sets}, which are defined by polynomial
inequalities.
By the Tarski--Seidenberg Theorem on quantifier elimination~\cite{Seidenberg,Tarski}, the class  of semialgebraic
sets is closed under projections and under images of polynomial maps.
A closed semialgebraic set is \demph{basic} if it is a finite intersection of sets of the form 
$\{x\mid f(x)\geq 0\}$, for $f$ a polynomial.

Motivating questions about convex algebraic geometry were raised in~\cite{orbitopes}.
A fundamental convex semialgebraic set is the cone of positive semidefinite matrices (the \demph{$\PSD$ cone}).
These are symmetric matrices with nonnegative eigenvalues.
The boundary of the $\PSD$ cone is (a connected component of) the determinant hypersurface and every face is
exposed. 
A \demph{spectrahedron} is an affine section $L\cap \PSD$ of this cone.
Write $A\defcolor{\succeq}0$ to indicate that $A \in \PSD$.
Parameterizing $L$ shows that a spectrahedron is defined by a \demph{linear matrix inequality},
\[
   \{x\in\RR^m \mid A_0+x_1A_1+\dotsb + x_mA_m \succeq 0\}\,,
\]
where $A_0,\dotsc,A_m$ are real symmetric matrices.

Images of spectrahedra under linear maps are \demph{spectrahedral shadows}.
Semidefinite programming provides efficient methods to optimize linear objective functions over spectrahedra
and their shadows, and a fundamental question is to determine if a given convex semialgebraic set may be
realized as a spectrahedron or as a spectrahedral shadow, and to give such a realization.
Scheiderer showed that the convex hull of a curve is a spectrahedral shadow~\cite{Sch12}, and 
recently showed that there are many convex semialgebraic sets which are not spectrahedra or their
shadows~\cite{Sch16}. 

Since the optimizer of a linear objective function lies in the boundary, convex
algebraic geometry also seeks to understand the boundary of a convex semialgebraic set.
This includes determining its faces and their inclusions, as well as
the Zariski closure of the boundary, called the \demph{algebraic boundary}.
This was studied for rational curves~\cite{Sinn,Vinzant} and for curves in $\RR^3$ by Ranestad
and Sturmfels~\cite{RS12}.
They showed that the algebraic boundary of a space curve $C$ consists of finitely many
tritangent planes and a ruled \demph{edge surface} composed of stationary bisecant lines.
A \demph{stationary bisecant} is a secant $\overline{x,y}$ to $C$ ($x,y\in C$) such that the tangent lines
$T_xC$ and $T_yC$ to $C$ at $x$ and $y$ meet.
For a general irreducible space curve of degree $d$ and genus $g$, the edge surface has degree
$2(d{-}3)(d{+}g{-}1)$~\cite{Johnsen,RS12}.

For example, suppose that $C$ is a general space quartic (see~\cite[Rem.~5.5]{Johnsen} or~\cite[Ex.~2.3]{RS12}).
This is the complete intersection of two real quadrics $P$ and $Q$, and has genus one by the
adjunction formula~\cite[Ex.~V.1.5.2]{hartshorne}.
Its edge surface has degree $2(4{-}3)(4{+}1{-}1)=8$ and is the union of four cones.
In the pencil of quadrics that contain $C$,  $sP+tQ$ for $[s,t]\in\PP^1$, four are singular
and are given by the roots of $\det(sP+tQ)$.
Here, the quadratic forms $P,Q$ are expressed as symmetric matrices. 
Each singular quadric is a cone and each line on that cone is a stationary bisecant of $C$.
A general point of $C$ lies on four stationary bisecants, one for each cone.

The union of two circles in different planes is also a space quartic, but it is not in general a complete
intersection (the complex points of a complete intersection are connected).
We therefore expect a different answer than for general space quartics.
We give a taste of that which is to come.

%%%%%%%%%%%%%%%%%%%%%%%%%%%%%%%%%%%%%%%%%%%%%%%%%%%%%%%%%%%%%%%%%%%%%%%%%%%%%%%%%
\begin{theorem}\label{Th:one}
 Let $C_1$ and $C_2$ be circles in $\RR^3$ lying in different planes.
 Their convex hull is a spectrahedron if and only if the 
 scheme $C_1\cap C_2$ has length 2.
 When the complex points of the circles are disjoint and neither is tangent to the plane of the other, the edge
 surface is irreducible and has degree eight. 
 Its rulings are parameterized by a smooth curve of genus one in $C_1\times C_2$.
 A general point of $C_1\cup C_2$ lies on two stationary bisecants.
\end{theorem}
%%%%%%%%%%%%%%%%%%%%%%%%%%%%%%%%%%%%%%%%%%%%%%%%%%%%%%%%%%%%%%%%%%%%%%%%%%%%%%%%%
\begin{proof}
 This is proven in Lemma~\ref{L:stBis}, and in
 Theorems~\ref{Th:edgeCurve},~\ref{Th:degreeEight} and ~\ref{Th:isSectrhedron}.
\end{proof}

%%%%%%%%%%%%%%%%%%%%%%%%%%%%%%%%%%%%%%%%%%%%%%%%%%%%%%%%%%%%%%%%%%%%%%%%%%%%%%%%%
\section{Stationary bisecants to two complex conics}\label{S:relaxation}
We study stationary bisecants and edge surfaces in the algebraic relaxation of our problem of two circles,
replacing circles in $\RR^3$ by smooth conics in $\defcolor{\PP^3}=\CC\PP^3$.

%%%%%%%%%%%%%%%%%%%%%%%%%%%%%%%%%%%%%%%%%%%%%%%%%%%%%%%%%%%%%%%%%%%%%%%%%%%%%%%%%

A conic $C$ in  $\PP^3$ spans a plane.
Let $C_1$ and $C_2$ be conics spanning different planes, $\Pi_1$ and $\Pi_2$, respectively.
A stationary bisecant is spanned by points $p\in C_1$ and $q\in C_2$ with $p\neq q$ whose tangent lines 
$T_pC_1$ and $T_qC_2$ meet.
Set $\defcolor{\ell}:=\Pi_1\cap\Pi_2$.

%%%%%%%%%%%%%%%%%%%%%%%%%%%%%%%%%%%%%%%%%%%%%%%%%%%%%%%%%%%%%%%%%%%%%%%%%%%%%%%%%
\begin{lemma}\label{L:stBis}
 A point $p\in C_1$ lies on two stationary bisecants unless the tangent line $T_pC_1$ meets
 $C_2$.
 If the tangent line meets $C_2$, then it is the unique stationary bisecant through $p$ unless
 $p\in C_2$ or $T_pC_1$ lies in the plane $\Pi_2$ of $C_2$.
 When $T_pC_1\subset\Pi_2$, the pencil of lines in $\Pi_2$ through $p$ are all stationary bisecants.
\end{lemma}
%%%%%%%%%%%%%%%%%%%%%%%%%%%%%%%%%%%%%%%%%%%%%%%%%%%%%%%%%%%%%%%%%%%%%%%%%%%%%%%%%

%%%%%%%%%%%%%%%%%%%%%%%%%%%%%%%%%%%%%%%%%%%%%%%%%%%%%%%%%%%%%%%%%%%%%%%%%%%%%%%%%
\begin{proof}
 See Fig.~\ref{F:stationaryBisecants} for reference.
%%%%%%%%%%%%%%%%%%%%%%%%%%%%%%%%%%%%%%%%%%%%%%%%%%%%%%%%%%%%%%%%%%%%%%%%%%%%%%%%%
\begin{figure}[htb]
\centering
  \begin{picture}(330,117)(-15,-10)
   \put(-15,-10){\includegraphics[height=117pt]{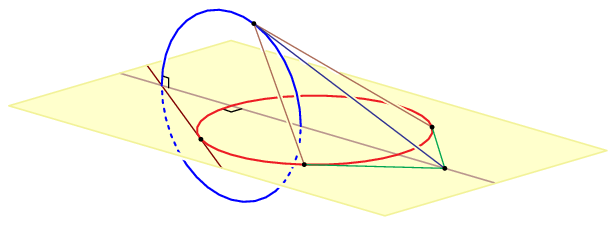}}
   \thicklines
   \put(179,87.8){{\color{white}\line(-3,-1){45}}}
   \put(179,88.2){{\color{white}\line(-3,-1){45}}}
%   \put(175,78.8){{\color{white}\line(-1,-1){12}}}
%   \put(175,79.2){{\color{white}\line(-1,-1){12}}}
   \put(218,62.3){{\color{white}\line(-2,-1){28}}}
   \put(218,62.7){{\color{white}\line(-2,-1){28}}}
   \put(30.5,27.8){{\color{white}\line(3,1){50}}}
   \put(30.5,28.2){{\color{white}\line(3,1){50}}}
   \thinlines
   \put(25,52){$\Pi_2$}
   \put(70,98){$C_1$}   \put(115,26){$C_2$}
   \put(123,101){$p$}  \put(228,21){$q$}  \put(247,13){$\ell$}
   \put(148,11){$r$}  \put(222,39){$s$}
   \put(181,86){stationary bisecants through $p$}
   \put(179,88){\vector(-3,-1){47}}
   \put(179,85){\vector(-1,-1){12}}
   \put(219,62){$T_p C_1$}
   \put(218,62.5){\vector(-2,-1){29}}
   \put(30.5,28){\vector(3,1){52}}
   \put(-14,19){tangent stationary}
   \put(-14, 9){bisecant}
  \end{picture}
 \caption{Stationary bisecants.}\label{F:stationaryBisecants}
\end{figure}
%%%%%%%%%%%%%%%%%%%%%%%%%%%%%%%%%%%%%%%%%%%%%%%%%%%%%%%%%%%%%%%%%%%%%%%%%%%%%%%%%
Consider the tangent line $T_pC_1$ for $p\in C_1$.
Either 
\[
   (i)\ T_pC_1\ \not\subset\ \Pi_2
   \qquad\mbox{or}\qquad
   (ii)\ T_pC_1\ \subset\ \Pi_2\,.
\]
In case $(i)$, let $q$ be the point where $T_pC_1$ meets $\Pi_2$.
There are further cases.
When $q\not\in C_2$, there are two tangents to $C_2$ that meet $q$, and
the lines through $p$ and each point of tangency ($r,s$ in Fig.~\ref{F:stationaryBisecants})  give two
stationary bisecants through $p$. 
If $q\in C_2$ and $p\neq q$, then the tangent line $T_pC_1$ is the only
stationary bisecant through $p$.

In case $(ii)$, the tangent line $T_pC_1$ meets every tangent to
$C_2$, and every line in $\Pi_2$ through $p$ (except $T_pC_2$ if $p\in C_2$) meets $C_2$  and is
therefore a stationary bisecant.
If $p\in C_2$, then the tangent line $T_pC_2$ is a limit of such lines.
\end{proof}
%%%%%%%%%%%%%%%%%%%%%%%%%%%%%%%%%%%%%%%%%%%%%%%%%%%%%%%%%%%%%%%%%%%%%%%%%%%%%%%%%

%%%%%%%%%%%%%%%%%%%%%%%%%%%%%%%%%%%%%%%%%%%%%%%%%%%%%%%%%%%%%%%%%%%%%%%%%%%%%%%%%
\begin{remark}
 When $C_1$ is tangent to the plane $\Pi_2$ at a point $p$, the pencil of lines in
 $\Pi_2$ through $p$ are \demph{degenerate stationary bisecants}.
 When $p\not\in C_2$, a general line in the pencil meets $C_2$ twice so that the map from $C_2$ to this pencil has
 degree two.
\end{remark}
%%%%%%%%%%%%%%%%%%%%%%%%%%%%%%%%%%%%%%%%%%%%%%%%%%%%%%%%%%%%%%%%%%%%%%%%%%%%%%%%%

Lines that meet $C_1$ and $C_2$ in distinct points are given by points $(p,q)\in C_1\times C_2$ with $p\neq q$. 
The \demph{edge curve $E$} is the Zariski closure of the set of points $(p,q)$ such that
$\overline{p,q}$ is a stationary bisecant.
As a smooth conic is isomorphic to $\PP^1$, the edge curve is a curve in $\PP^1\times\PP^1$.
Subvarieties of products of projective spaces have a multidegree (see~\cite[\S~2]{multidegree}).
For a curve $C$ in $\PP^1\times\PP^1$, this becomes its bidegree $(a,b)$, where $a$ is the number of points in the
intersection of $C$ with $\PP^1\times\{q\}$ for $q$ general and $b$ the number of points in the intersection of $C$
with $\{p\}\times\PP^1$ for $p$ general.
As $\PP^1\times\{q\}$ has bidegree $(0,1)$ and $\{p\}\times\PP^1$ bidegree $(1,0)$, the intersection pairing on
curves in $\PP^1\times\PP^1$, expressed in terms of bidegree, is 
 \begin{equation}\label{Eq:IntersectionPairing}
   (a,b)\cdot (c,d)\ =\ ad+bc\ \in\ \ZZ\,.
 \end{equation}
A curve of bidegree $(a,b)$ is defined in homogeneous coordinates $([s:t],[u:v])$ for $\PP^1\times\PP^1$ by a
bihomogeneous polynomial that has degree $a$ in $s,t$ and $b$ in $u,v$.

%%%%%%%%%%%%%%%%%%%%%%%%%%%%%%%%%%%%%%%%%%%%%%%%%%%%%%%%%%%%%%%%%%%%%%%%%%%%%%%%%
\begin{theorem}\label{Th:edgeCurve}
 The edge curve $E$ has bidegree $(2,2)$.
\end{theorem}
%%%%%%%%%%%%%%%%%%%%%%%%%%%%%%%%%%%%%%%%%%%%%%%%%%%%%%%%%%%%%%%%%%%%%%%%%%%%%%%%%

%%%%%%%%%%%%%%%%%%%%%%%%%%%%%%%%%%%%%%%%%%%%%%%%%%%%%%%%%%%%%%%%%%%%%%%%%%%%%%%%%
\begin{proof}
 In the projection to $C_1$, two points of $E$ map to a general point $p\in C_1$, 
 by Lemma~\ref{L:stBis}.
 Thus the intersection number of $E$ with $\{p\}\times C_2$ is 2 and vice-versa for
 $C_1\times\{q\}$, for general $q\in C_2$.
 Consequently, $E$ has bidegree $(2,2)$.

 We compute the defining equation of $E$ to give a second proof.
 This begins with a parameterization of the conics.
 Let $f_{i,0},\dotsc,f_{i,3}\in H^0(\PP^1,\cO(2))$ for $i=1,2$ be two quadruples of homogeneous quadrics that each
 span $H^0(\PP^1,\cO(2))$.
 Each quadruple gives a map $f_i\colon\PP^1\to\PP^3$ whose image is a conic $C_i$.
 The plane $\Pi_i$ of $C_i$ is defined by the linear relation among $f_{i,0},\dotsc,f_{i,3}$, and we assume that
 $\Pi_1 \neq \Pi_2$. 

 In coordinates, if $[s:t]\in\PP^1$, then the image 
\[
  f_i[s:t]\ =\ [f_{i,0}(s,t):f_{i,1}(s,t):f_{i,2}(s,t):f_{i,3}(s,t)]\,
\]
 is the corresponding point of $C_i$.
 Its tangent line is spanned by 
 $\partial_s f_i$ and $\partial_t f_i$, where $\partial_x:=\frac{\partial}{\partial x}$, as 
 $s\partial_s f + t \partial_t f=2f$, for a homogeneous quadric $f$.
 The points $f_1[s:t]$ and $f_2[u:v]$ span a stationary bisecant when their tangents meet.
 Equivalently, when
 \begin{equation}\label{Eq:edgeCurve}
  E(s,t,u,v)\ :=\ 
    \det \begin{pmatrix}
         \partial_s f_{1,0} &\ \partial_s f_{1,1} &\ \partial_s f_{1,2} &\ \partial_s f_{1,3} \\\rule{0pt}{11pt}
         \partial_t f_{1,0} &\ \partial_t f_{1,1} &\ \partial_t f_{1,2} &\ \partial_t f_{1,3} \\\rule{0pt}{11pt}
         \partial_u f_{2,0} &\ \partial_u f_{2,1} &\ \partial_u f_{2,2} &\ \partial_u f_{2,3} \\\rule{0pt}{11pt}
         \partial_v f_{2,0} &\ \partial_v f_{2,1} &\ \partial_v f_{2,2} &\ \partial_v f_{2,3} 
    \end{pmatrix}\ =\ 0\,.
 \end{equation}
 As the first two rows have bidegree $(1,0)$ and the second two have bidegree
 $(0,1)$, this form $E(s,t,u,v)$ has bidegree $(2,2)$.
\end{proof}
%%%%%%%%%%%%%%%%%%%%%%%%%%%%%%%%%%%%%%%%%%%%%%%%%%%%%%%%%%%%%%%%%%%%%%%%%%%%%%%%%

%%%%%%%%%%%%%%%%%%%%%%%%%%%%%%%%%%%%%%%%%%%%%%%%%%%%%%%%%%%%%%%%%%%%%%%%%%%%%%%%%
\begin{example}\label{Ex:edge_curves}
 Suppose that  $C_1$ and $C_2$ are the unlinked unit circles where 
 $C_1$ is centered at the origin and lies in the $xy$-plane and $C_2$ is centered at $(3,0,0)$ and lies
 in the $xz$-plane.
 If we choose homogeneous coordinates $[X_0:X_1:X_2:X_3]$ for $\PP^3$ where
 $(x,y,z)=\frac{1}{X_0}(X_1,X_2,X_3)$, then these admit parametrizations
\[
   [s:t]\ \longmapsto\ [s^2+t^2: s^2-t^2 : 2st:0]
   \quad\mbox{and}\quad
   [u:v]\ \longmapsto\ [u^2+v^2: 2u^2+4v^2 : 0: 2uv]\,.
\]
 Dividing the determinant~\eqref{Eq:edgeCurve} by $-16$ gives the equation for the edge curve $E$,
\[
  s^2u^2\ -\ 3 s^2 v^2\ -\ 3t^2 u^2\ +\ 5 t^2 v^2\,,
\]
%
%   Warning, this is a little different from the two circle roller entitled unLinked; the circles here 
%    are further apart, as Frank wants to make a good picture of the inside.
%
which is irreducible.
We draw $E$ below in the window $|s/t|,|u/v|\leq 5$ in $\RR\PP^1\times\RR\PP^1$.
 \begin{equation}\label{Eq:unLinked_wide_curve}
   \raisebox{-30pt}{\includegraphics[height=70pt]{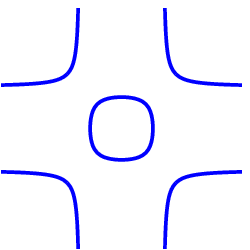}}
 \end{equation}

\end{example}
%%%%%%%%%%%%%%%%%%%%%%%%%%%%%%%%%%%%%%%%%%%%%%%%%%%%%%%%%%%%%%%%%%%%%%%%%%%%%%%%%

The Zariski closure of the union of all stationary bisecants is the ruled \demph{edge surface $\cE$}.
By Lemma~\ref{L:stBis}, a general point of one of the conics lies on two stationary bisecants.
Therefore, each conic is a curve of self-intersections of $\cE$, and the multiplicity of $\cE$ at a general point
of a conic is $2$.

%%%%%%%%%%%%%%%%%%%%%%%%%%%%%%%%%%%%%%%%%%%%%%%%%%%%%%%%%%%%%%%%%%%%%%%%%%%%%%%%%
\begin{theorem}\label{Th:degreeEight}
 The edge surface $\cE$ has degree eight when $C_1\cap C_2=\emptyset$ and neither is tangent to the plane of the
 other.  
\end{theorem}
%%%%%%%%%%%%%%%%%%%%%%%%%%%%%%%%%%%%%%%%%%%%%%%%%%%%%%%%%%%%%%%%%%%%%%%%%%%%%%%%%

%%%%%%%%%%%%%%%%%%%%%%%%%%%%%%%%%%%%%%%%%%%%%%%%%%%%%%%%%%%%%%%%%%%%%%%%%%%%%%%%%
\begin{proof}
 The line $\ell=\Pi_1\cap\Pi_2$ meets each conic in two points and therefore meets $\cE$ in
 at least four points. 
 Any other point $\defcolor{r}\in\ell\cap\cE$ lies on a stationary bisecant $m$ between a point
 $p$ of $C_1$ and a point $q$ of $C_2$. 
 As $p,r\in \Pi_1$, we have $m\subset\Pi_1$, and similarly $m\subset\Pi_2$.
 Thus $m=\ell$, but $\ell$ is not a stationary bisecant, a contradiction.

 Each of the four points of $\ell\cap\cE$ has multiplicity two on $\cE$ by
 Lemma~\ref{L:stBis} and the observation preceding the statement of the theorem.
 Thus, $\cE$ has degree eight.

 We give a second proof.
 Let \defcolor{$m$} be a general line that meets $\cE$ transversally.
 The points of $m\cap\cE$ lie on stationary bisecants that meet $m$.
 We count these using intersection theory.
 Let $\defcolor{M}\subset C_1\times C_2$ be the curve whose points are pairs $(p,q)$ such that the secant
 line spanned by $p$ and $q$ meets $m$.
 Stationary bisecants that meet $m$ are 
 points of intersection of $M$ and the edge curve $E$.
 We compute the bidegree of $M$.

 Fix a point $p\in C_1$ with $p\not\in\Pi_2$.
 Secant lines through $p$ rule the cone over $C_2$ with vertex $p$.
 As this cone meets $m$ in two points, we have $\deg(M\cap\{p\}\times C_2)=2$.
 The symmetric argument with a point of $C_2$ shows that $M$ has bidegree $(2,2)$.
 By~\eqref{Eq:IntersectionPairing},  $M$ meets $E$ in 
 $(2,2)\cdot(2,2)=4+4=8$ points.
 This proves the theorem. 
\end{proof}
%%%%%%%%%%%%%%%%%%%%%%%%%%%%%%%%%%%%%%%%%%%%%%%%%%%%%%%%%%%%%%%%%%%%%%%%%%%%%%%%%

The arguments in this proof using intersection theory are similar to arguments used in the
contributions~\cite{multidegree,congruences} in this volume.

%%%%%%%%%%%%%%%%%%%%%%%%%%%%%%%%%%%%%%%%%%%%%%%%%%%%%%%%%%%%%%%%%%%%%%%%%%%%%%%%%
\begin{remark}
 Each irreducible component $C$ of the edge curve $E$ gives an algebraic family of stationary bisecants and an
 irreducible component \defcolor{$\cC$} of the edge surface $\cE$.
 If $C$ has bidegree $(a,b)$, then the corresponding component $\cC$ of $\cE$ has degree 
 at most $(2,2)\cdot(a,b)=2(a+b)$.
 This is not an equality when the intersection $M\cap E$ has a basepoint 
 or when the general point of $\cC$ contains two stationary bisecants.
 This occurs when one circle is tangent to the plane of the other and there are one or more
 components of degenerate stationary bisecants.
\end{remark}
%%%%%%%%%%%%%%%%%%%%%%%%%%%%%%%%%%%%%%%%%%%%%%%%%%%%%%%%%%%%%%%%%%%%%%%%%%%%%%%%%

%%%%%%%%%%%%%%%%%%%%%%%%%%%%%%%%%%%%%%%%%%%%%%%%%%%%%%%%%%%%%%%%%%%%%%%%%%%%%%%%%
\begin{example}\label{Ex:edge_curves2}
 The real points of the edge curve~\eqref{Eq:unLinked_wide_curve} of Example~\ref{Ex:edge_curves} 
 had two connected components (the picture showed a patch of $\RR\PP^1\times\RR\PP^1$).
 Thus, the set of real points of the edge surface has two components.
 Stationary bisecants corresponding to the oval in the center of~\eqref{Eq:unLinked_wide_curve} lie along
 the convex hull, which shown on the left below.
 The others bound a nonconvex set that lies inside the convex hull.
 We display it in an expanded view on the right below.
%
%  These need to be redrawn. They are for the second circle at [5/2,0,0] and the rulings on the right are
%  messed up
%
\[
   \includegraphics[height=85pt]{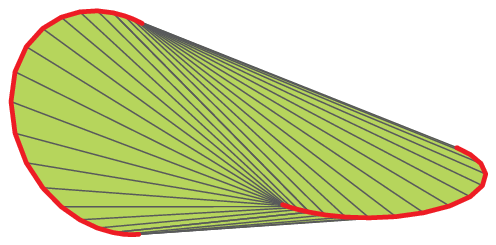}\qquad\qquad
   \includegraphics[height=85pt]{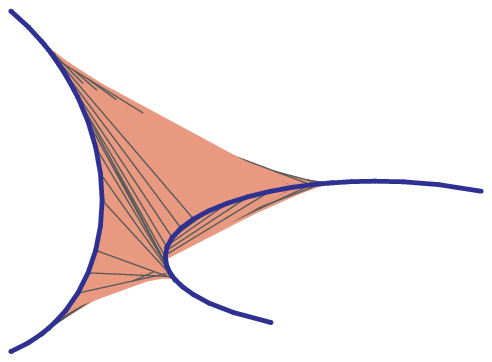}
\]
The planes of the circles meet in the $x$-axis.
For sufficiently small $\epsilon>0$, the line defined by $y=z=\epsilon$ meets $\cE$ transversally.
Near each point of a circle lying on the $x$-axis it meets $\cE$ in two points, one for each of the two 
families of stationary bisecants passing through the nearby arc of the circle.
These eight points are real.
%
%  When the pictures are redrawn, Frank will include this line and its four points of intersection.
%
\end{example}
%%%%%%%%%%%%%%%%%%%%%%%%%%%%%%%%%%%%%%%%%%%%%%%%%%%%%%%%%%%%%%%%%%%%%%%%%%%%%%%%%

%%%%%%%%%%%%%%%%%%%%%%%%%%%%%%%%%%%%%%%%%%%%%%%%%%%%%%%%%%%%%%%%%%%%%%%%%%%%%%%%%
A curve of bidegree $(2,2)$ on $\PP^1\times\PP^1$ has arithmetic genus one, by the adjunction
formula~\cite[Ex.~V.1.5.2]{hartshorne}.
If smooth, then it is an irrational genus one curve.
Another way to see this is that the projection to a $\PP^1$ factor is two-to-one, except over the branch 
points, of which there are four, counted with multiplicity.
Indeed, writing its defining equation as a quadratic form in the variables $(u,v)$ for the second $\PP^1$ factor, its
coefficients are quadratic forms in the variables $s,t$ of the first $\PP^1$.
The projection to the first has branch points where the discriminant vanishes, which is a quartic form.
By elementary topology, a double cover of $\CC\PP^1$ with four branch points has Euler characteristic zero, again
implying that it has genus one.

%%%%%%%%%%%%%%%%%%%%%%%%%%%%%%%%%%%%%%%%%%%%%%%%%%%%%%%%%%%%%%%%%%%%%%%%%%%%%%%%%
\begin{lemma}\label{L:any4points}
 For every set $S$ of four points of $C_1$, there is a conic $C_2$ such that the projection to $C_1$ of the
 edge curve is branched over $S$.
\end{lemma}
%%%%%%%%%%%%%%%%%%%%%%%%%%%%%%%%%%%%%%%%%%%%%%%%%%%%%%%%%%%%%%%%%%%%%%%%%%%%%%%%%

%%%%%%%%%%%%%%%%%%%%%%%%%%%%%%%%%%%%%%%%%%%%%%%%%%%%%%%%%%%%%%%%%%%%%%%%%%%%%%%%%
\begin{proof}
 Let $p$ be the point of intersection of two of the tangents to $C_1$ at points of $S$ and $q$ be the point of
 intersection of the other two tangents (see Fig.~\ref{F:jInvariant}).
 Since the tangent $T_sC_1$ at any point $s\in S$ meets $C_2$ (in one of the points $p$ or $q$),
 Lemma~\ref{L:stBis} implies this is the unique stationary bisecant involving the point $s$.
 Thus, the points of $S$ are branch points of the projection to $C_1$ of the edge curve.
\end{proof}
%%%%%%%%%%%%%%%%%%%%%%%%%%%%%%%%%%%%%%%%%%%%%%%%%%%%%%%%%%%%%%%%%%%%%%%%%%%%%%%%%

\begin{figure}[htb]
 \centering
 \begin{picture}(265,100)
    \put(0,0){\includegraphics[height=100pt]{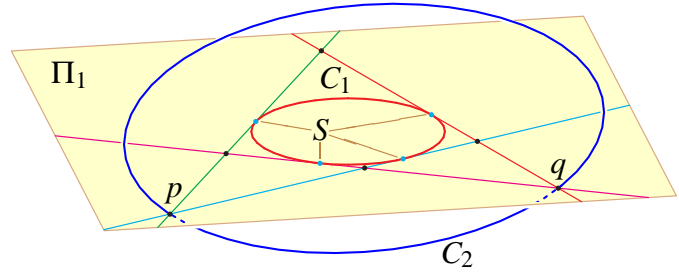}}
     \put(120,66){$C_1$}  \put(18,68){$\Pi_1$}
     \put(61,25){$p$}      \put(207,34){$q$}

     \put(117,46.5){$S$}

     \put(119.7,39.3){{\color{brown}\line(0,1){6}}}
     \put(149,39.8333){{\color{brown}\line(-3,1){25.6}}}
     \put(160,55.8333){{\color{brown}\line(-6,-1){36.6}}}
     \put( 97.5,53.1666){{\color{brown}\line(6,-1){18.9}}}
     \put(166, 0){$C_2$}
 \end{picture}
 \caption{Conic giving specified branch points.}\label{F:jInvariant}
\end{figure}

%%%%%%%%%%%%%%%%%%%%%%%%%%%%%%%%%%%%%%%%%%%%%%%%%%%%%%%%%%%%%%%%%%%%%%%%%%%%%%%%%
\begin{remark}\label{R:families}
 There are three families of conics $C_2$ giving an edge curve branched over $S$.
 These correspond to the three partitions of $S$ into two parts of size two.
 Each partition determines two points $p,q$ on the plane $\Pi_1$ of $C_1$ where the tangent lines at the points in
 each part meet.  
 The corresponding family is the collection of conics $C_2$ that meet $\Pi_1$ transversally in $p$ and $q$.

 If both $C_1$ and $S\subset C_1$ are real and we choose an affine $\RR^3$ containing the points $p$ and
 $q$, then we may choose $C_2$ to be a circle.
\end{remark}
%%%%%%%%%%%%%%%%%%%%%%%%%%%%%%%%%%%%%%%%%%%%%%%%%%%%%%%%%%%%%%%%%%%%%%%%%%%%%%%%%

The isomorphism  class of a complex smooth genus one curve is determined by its
$j$-invariant~\cite[\S~IV.4]{hartshorne}. 
This may be computed from the branch points $S$ of any degree two map to $\PP^1$.
Explicitly, if we choose coordinates on $\PP^1$ so that the branch points $S$ are 
$\{0,1,\lambda,\infty\}$, then the $j$-invariant is
\[
    2^8\cdot\frac{( \lambda^2-\lambda+1)^3}{\lambda^2(\lambda-1)^2}\,.
\]
We have the following corollary of Lemma~\ref{L:any4points}.

%%%%%%%%%%%%%%%%%%%%%%%%%%%%%%%%%%%%%%%%%%%%%%%%%%%%%%%%%%%%%%%%%%%%%%%%%%%%%%%%%%%%%%%%%%%%%%%%%%%%%%%%%%%%
\begin{theorem}
 For every conic $C_1$ and every $J\in\CC$, there is a conic $C_2$ such that the edge curve has
 $j$-invariant $J$.
 When $C_1$ and $S\subset C_1$ are real, $C_2$ may be a circle.
\end{theorem}
%%%%%%%%%%%%%%%%%%%%%%%%%%%%%%%%%%%%%%%%%%%%%%%%%%%%%%%%%%%%%%%%%%%%%%%%%%%%%%%%%%%%%%%%%%%%%%%%%%%%%%%%%%%%

We now classify the possible edge curves $E$ to a pair of conics $C_1$ and $C_2$ lying in distinct
planes $\Pi_1$ and $\Pi_2$.
By Lemma~\ref{L:fivePossibilities}, every component of $E$ is reduced.
If $E=F\cup G$ is reducible, then 
we have that 
$(2,2)=\mbox{bidegree}(F)+\mbox{bidegree}(G)$.
Thus, the bidegrees of the components of $E$ form a partition of $(2,2)$.
If $E$ is irreducible, then either it is smooth of genus one or singular of arithmetic genus one and hence
rational.
Any curve of bidegree $(1,a)$ or $(a,1)$ is rational.
Table~\ref{T:components} gives the different possibilities, along with pictures of a real curve.
%%%%%%%%%%%%%%%%%%%%%%%%%%%%%%%%%%%%%%%%%%%%%%%%%%%%%%%%%%%%%%%%%%%%%%%%%%%%%%%%%
\begin{table}[htb]
 \begin{tabular}{|c|c|c|c|}\hline
  smooth (generic) & nodal rational & cuspidal & \small$2(1,0)+2(0,1)$\rule{0pt}{10pt}\\
   \includegraphics[width=60pt]{pictures/unLinked_wide_curve}\rule{0pt}{62pt}&
   \includegraphics[width=60pt]{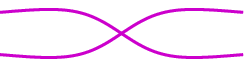}&
  \includegraphics[width=60pt]{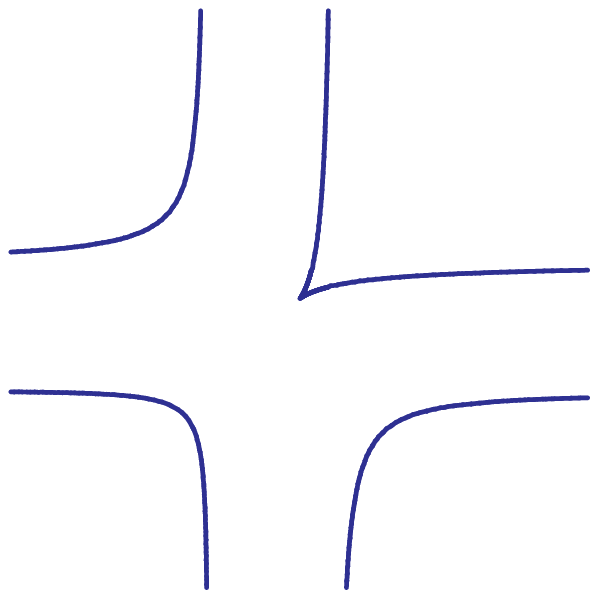}&
   \includegraphics[width=60pt]{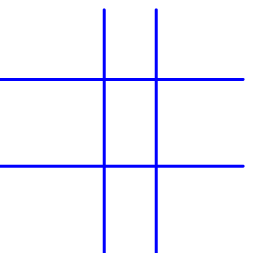}\\\hline
\end{tabular}

 \begin{tabular}{|c|c|c|c|c|}\hline
   $2(1,1)$ &$(2,1)+(0,1)$&$(2,1)+(0,1)$& $(1,1)+(0,1)$&  $(1,1)+(0,1)$ \rule{0pt}{10pt}\\
   &&&$\ \ \ +(1,0)$&$\ \ \ +(1,0)$\\
   \includegraphics[width=60pt]{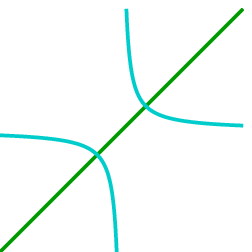}\rule{0pt}{62pt}&
   \includegraphics[width=60pt]{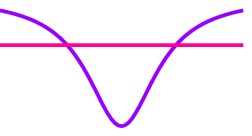}&
   \includegraphics[width=60pt]{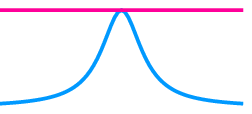}&
   \includegraphics[width=60pt]{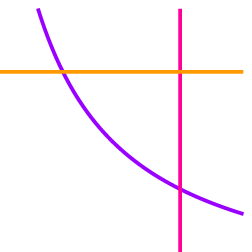}&
   \includegraphics[width=60pt]{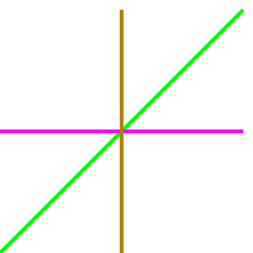}\\\hline
 \end{tabular}
\caption{Types of $(2,2)$-curves.}\label{T:components}
\end{table}
%%%%%%%%%%%%%%%%%%%%%%%%%%%%%%%%%%%%%%%%%%%%%%%%%%%%%%%%%%%%%%%%%%%%%%%%%%%%%%%%%

%%%%%%%%%%%%%%%%%%%%%%%%%%%%%%%%%%%%%%%%%%%%%%%%%%%%%%%%%%%%%%%%%%%%%%%%%%%%%%%%%
\begin{theorem}\label{Th:edgeCurves}
 All types of $(2,2)$-curves of Table~\ref{T:components} occur as the edge curve of a pair of conics
 $C_1,C_2$ lying in distinct planes except a curve with a cusp and a reducible curve 
 $2(1,0)+2(0,1)$ with four components.
\end{theorem}
%%%%%%%%%%%%%%%%%%%%%%%%%%%%%%%%%%%%%%%%%%%%%%%%%%%%%%%%%%%%%%%%%%%%%%%%%%%%%%%%%

For existence, see Tables~\ref{T:GalleryI},~\ref{T:GalleryII} and~\ref{T:GalleryIII}, which display edge curves of
two circles in all possible configurations.
We rule out edge curves with a cusp and reducible edge curves of type $2(1,0)+2(0,1)$.
We first analyze the singularities of edge curves.

%%%%%%%%%%%%%%%%%%%%%%%%%%%%%%%%%%%%%%%%%%%%%%%%%%%%%%%%%%%%%%%%%%%%%%%%%%%%%%%%%
\begin{lemma}\label{L:fivePossibilities}
 The edge curve $E$ is reduced.
 A point $(p,q)\in C_1\times C_2$ is a singular point of $E$ only if 
 $p=q$ or $T_pC_1\subset\Pi_2$ or $T_qC_2\subset\Pi_1$.
 There are five possibilities for $p,q$ and the tangents, up to interchanging the conics $C_1$ and $C_2$.
\renewcommand{\theenumi}{\roman{enumi}}%
\begin{enumerate}
 \item $p=q$ and the tangent to each conic at $p$ does not lie in the plane of the other. 
 \item $p=q$ with $T_p C_1\subset\Pi_2$, but $T_q C_2\not\subset\Pi_1$.
 \item $p=q$ with both $T_p C_1\subset\Pi_2$ and $T_q C_2\subset\Pi_1$.
 \item $p\neq q$ and $T_p C_1\subset\Pi_2$, but $T_q C_2\not\subset\Pi_1$.
        Then $p\in T_q C_2$ is a stationary bisecant.
 \item $p\neq q$ and $T_p C_1=T_q C_2$ is $\Pi_1\cap\Pi_2$, and is a stationary bisecant.
\end{enumerate}
\end{lemma}
%%%%%%%%%%%%%%%%%%%%%%%%%%%%%%%%%%%%%%%%%%%%%%%%%%%%%%%%%%%%%%%%%%%%%%%%%%%%%%%%%

%%%%%%%%%%%%%%%%%%%%%%%%%%%%%%%%%%%%%%%%%%%%%%%%%%%%%%%%%%%%%%%%%%%%%%%%%%%%%%%%%
\begin{proof}
 Let $(p,q)\in C_1\times C_2$ be a point on a curve $E$ of bidegree $(2,2)$.
 If the fiber of $E$ in one of the projections from $(p,q)$, say to $C_2$, has exactly two points, then $E$
 is smooth at $(p,q)$.
 Indeed, as $E$ is a $(2,2)$ curve,  $E\cap(C_1\times\{q\})$ is either $C_1\times\{q\}$ or one double or two simple 
 points, and if two, then $E$ is smooth at each point.

 Consequently, there are three possibilities for points of $E$ in the fibers of the projections to $C_1$ and
 $C_2$ containing a singular point $(p,q)$.
 Either
\begin{enumerate}
 \item $(p,q)$ is the only point of $E$ in both fibers, 
 \item $(p,q)$ is the only point in one fiber and the other fiber is a component of $E$, or
 \item both fibers are components of $E$.
\end{enumerate}
 In Case 2, $E$ has at least one component with either linear bidegree $(1,0)$ or $(0,1)$, and in Case 3, it has at  
 least one component with each linear bidegree.

 Now let $E$ be the edge curve, which is smooth at any point $(p,q)$ where there is another point in one of the two
 fibers of projections to $C_i$. 
 Lemma~\ref{L:stBis} implies that there are two points in $E$ over a general point of either conic, so every
 component of $E$ is smooth at a generic point and therefore $E$ is reduced.
 By Lemma~\ref{L:stBis} and the analysis above, a point $(p,q)\in E$ is singular if and only if both tangents meet
 the other conic for otherwise there is a second point in one of the fibers. 

 If $T_p C_1\subset\Pi_2$, then every line in $\Pi_2$ through $p$ is a stationary bisecant,
 so $E$ contains $\{p\}\times C_2$, which has bidegree $(1,0)$. 
 If $T_qC_2\subset\Pi_1$, then as before $E$ contains $C_1\times\{q\}$, which has bidegree $(0,1)$.
 If neither occurs, but $E$ is singular at $(p,q)$, then we are in Case $(i)$.
 When $p=q$ and we are not in Case $(i)$, then, up to interchanging the indices 1 and 2, we are in either Case
 $(ii)$ or $(iii)$. 
 When $p\neq q$, so that one circle is tangent to the plane of the other, then we are in either Case $(iv)$
 or $(v)$.
\end{proof}
%%%%%%%%%%%%%%%%%%%%%%%%%%%%%%%%%%%%%%%%%%%%%%%%%%%%%%%%%%%%%%%%%%%%%%%%%%%%%%%%%

%%%%%%%%%%%%%%%%%%%%%%%%%%%%%%%%%%%%%%%%%%%%%%%%%%%%%%%%%%%%%%%%%%%%%%%%%%%%%%%%%
\begin{proof}[of Theorem~\ref{Th:edgeCurves}]
 We need only to rule out that the edge curve $E$ has type $2(1,0)+2(0,1)$ or has a cusp.
 By Lemma~\ref{L:fivePossibilities}, $E$ has a component $\{p\}\times C_2$ of bidegree $(1,0)$ exactly
 when $C_1$ is tangent to the plane $\Pi_2$ at the point $p$.
 Since $\Pi_1\neq\Pi_2$, there is at most one such point of tangency on $C_1$, so  $E$ has at most one
 component of bidegree $(1,0)$ and the same is true  for a component of bidegree $(0,1)$.
 Thus the type $2(1,0)+2(0,1)$ cannot occur for an edge curve.

 We show that if $(p,q)\in E$ is a singular point in Case $(i)$ of Lemma~\ref{L:fivePossibilities}, 
 then $E$ has a node at $(p,q)$, ruling out a cusp and
 completing the proof.

 Suppose that $p=q$ and the tangents to each conic at
 $p$ do not lie in the plane of the other. 
 Choose coordinates $x,y,z,w$ for $\PP^3$ so that $\Pi_1$ is the plane $z=0$, $\Pi_2$ is the plane $x=0$, 
 $T_pC_1$ is the line $y=z=0$, $T_pC_2$ is $x=y=0$, and $p=[0:0:0:1]$.
 Then we may choose parametrizations near $p$ for $C_1$ and $C_2$ of the form
 \begin{equation}\label{Eq:suParametrization}
  \begin{array}{rcl}
    C_1\ \colon\ s& \longmapsto& [\,s+as^2\,:\, bs^2\,:\,0\,:\, 1+cs+ds^2\,]\\\rule{0pt}{12pt}
    C_2\ \colon\ u& \longmapsto& [\,0\,:\, \beta u^2\,:\,u + \alpha u^2\,:\, 1+\gamma u+\delta u^2\,]\,,
  \end{array}
 \end{equation}
for some $a,b,c,d,\alpha,\beta,\gamma,\delta\in\CC$ where $b\beta\neq 0$.
The edge curve is defined by 
 \begin{multline}
   \det \begin{pmatrix}
     s+as^2 & bs^2 & 0 &  1+cs+ds^2\\
     1+2as\;  & 2bs  & 0 &  c+2ds\\
     0 & \beta u^2 & u + \alpha u^2 &  \;1+\gamma u+\delta u^2\;\\
     0 & \;2\beta u\;  & \;1 + 2\alpha u\;  &  \gamma + 2\delta u \end{pmatrix}
  \\\rule{0pt}{12pt}
  \ =\ 
    \left(\beta(ac-d)-b(\alpha\gamma-\delta)\right)s^2u^2  -2b\alpha s^2u + 2a\beta su^2
    +\beta u^2-bs^2\,.\label{Eq:newEdgeCurve}
 \end{multline}
Indeed, the matrix has rows $f_1(s), f'_1(s), f_2(u), f'_2(u)$, where $f_i$ is the parameterization of
$C_i$~\eqref{Eq:suParametrization}. 
The determinant vanishes when the tangent to $C_1$ at $f_1(s)$ meets the tangent to $C_2$ at $f_2(u)$.
The terms of lowest order 
in~\eqref{Eq:newEdgeCurve}, $\beta u^2 - b s^2$, have distinct linear factors
when $b\beta\neq 0$. 
Thus $E$ has a node when $s=u=0$, which is $(p,p)$.
\end{proof}
%%%%%%%%%%%%%%%%%%%%%%%%%%%%%%%%%%%%%%%%%%%%%%%%%%%%%%%%%%%%%%%%%%%%%%%%%%%%%%%%%

%%%%%%%%%%%%%%%%%%%%%%%%%%%%%%%%%%%%%%%%%%%%%%%%%%%%%%%%%%%%%%%%%%%%%%%%%%%%%%%%%
\section{Convex hull of two circles in $\RR^3$}\label{S:circles}

We classify the relative positions of two circles in $\RR^3$
and show that the combinatorial type of the face lattice of their convex hull depends only upon their relative
position. 
This relative position is determined by the combinatorial type of the face lattice and the real geometry of the edge
curve. 
We use this classification to determine when the convex hull of two circles is a spectrahedron.

%%%%%%%%%%%%%%%%%%%%%%%%%%%%%%%%%%%%%%%%%%%%%%%%%%%%%%%%%%%%%%%%%%%%%%%%%%%%%%%%%
Let $C_1,C_2$ be circles in $\RR^3$ lying in distinct planes $\Pi_1$ and $\Pi_2$, respectively. 
The intersection $C_1\cap \Pi_2$ in $\CC\PP^3$ is either two real points,
two complex conjugate points, or $C_1$ is tangent to $\Pi_2$ at a single real point.
Let $\defcolor{m_1}$ be the number of real points in this intersection, and the same for \defcolor{$m_2$}.
Order the circles so that $m_1\geq m_2$, and call  $[m_1,m_2]$ the 
\demph{intersection type} of the pair of circles.

The configuration of the circles is determined by the order of their points along the line 
$\defcolor{\ell}:=\Pi_1\cap\Pi_2\subset\RR^3$.
For example, $C_1$ and $C_2$ have \demph{order type $(1,2,1,2)$ along $\ell$} when they have  
intersection type $[2,2]$ and meet $\ell$ in distinct points which alternate.
If $C_1\cap C_2\neq \emptyset$, then we write \demph{$S$} for that shared point. 
For example, if $C_1$ meets $\ell$ in two real points with $C_2$ tangent to $\ell$ at one,
then this pair has order type \defcolor{$(1,S)$}.
The intersection type may be recovered from the order type.

A further distinction is necessary for intersection type $[0,0]$, when both circles meet $\ell$ in two complex
conjugate points.
In $\CC\PP^3$ either $C_1\cap C_2=\emptyset$ or %else
$C_1\cap\ell=C_2\cap\ell$.
Write \defcolor{$\emptyset$} for the order type in the first case and \defcolor{$(2c)$} in the second. 
By Lemma~\ref{L:fivePossibilities}, the edge curve is smooth in order type $\emptyset$ and singular in
order type $(2c)$.

%%%%%%%%%%%%%%%%%%%%%%%%%%%%%%%%%%%%%%%%%%%%%%%%%%%%%%%%%%%%%%%%%%%%%%%%%%%%%%%%%%%%%%%%%%%%%%%%
\begin{lemma}
 There are fifteen possible order types of two circles in $\RR^3$.
\end{lemma}
%%%%%%%%%%%%%%%%%%%%%%%%%%%%%%%%%%%%%%%%%%%%%%%%%%%%%%%%%%%%%%%%%%%%%%%%%%%%%%%%%%%%%%%%%%%%%%%%

%%%%%%%%%%%%%%%%%%%%%%%%%%%%%%%%%%%%%%%%%%%%%%%%%%%%%%%%%%%%%%%%%%%%%%%%%%%%%%%%%%%%%%%%%%%%%%%%
\begin{proof}
See Tables~\ref{T:GalleryI},~\ref{T:GalleryII}, and~\ref{T:GalleryIII} for the order types of circles and
their convex hulls. 
	
The possible intersection types are $[0,0]$, $[1,0]$, $[1,1]$, $[2,0]$, $[2,1]$, and $[2,2]$.
For $[0,0]$, we noted two order types, and intersection types $[1,0]$ and $[2,0]$ each admit one order type, namely
$(1)$ and $(1,1)$, respectively. 
	
For $[1,1]$, each circle $C_i$ is tangent to $\ell$ at a point $p_i$. 
Either $p_1 \neq p_2$ or $p_1= p_2$, so there are two order types, $(1,2)$ and $(S)$.
	
For $[2,1]$, the line $\ell$ is secant to circle $C_1$ and $C_2$ is tangent to $\ell$ at a point $p_2$.
Either $p_2$ is in the exterior of $C_1$ or it lies on $C_1$ or it is interior to $C_1$.
These give three order types $\ell$, $(1,1,2)$, $(1,S)$, and $(1,2,1)$, respectively.
	
Finally, for $[2,2]$ there are three order types when all four point are distinct, 
$(1,1,2,2)$, $(1,2,1,2)$, and $(1,2,2,1)$.
When one point is shared, we have $(1,2,S)$ or $(1,S,2)$.
Finally, both points may be shared, giving $(S,S)$.  
\end{proof}

%%%%%%%%%%%%%%%%%%%%%%%%%%%%%%%%%%%%%%%%%%%%%%%%%%%%%%%%%%%%%%%%%%%%%%%%%%%%%%%%%
\begin{table}[htb]
  %
  %  Frank rearranged this table by the complexity of the (2,2)-curves
  %
  \begin{tabular}{|c|c|c|}\hline
%%%%%%%%%%%%%%%%%  The first row are those with smooth (2,2) curves  
   \includegraphics[width=100pt]{pictures/bothExtreme}&
   \includegraphics[width=95pt]{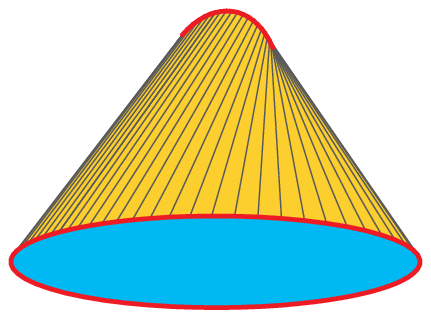}&
   \includegraphics[width=90pt]{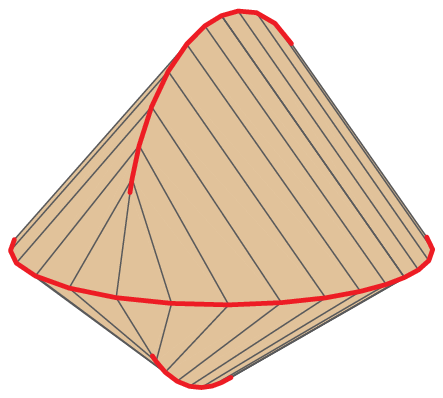}\rule{0pt}{84pt}\\
   $[0,0]$ \quad $\emptyset$     & $[2,0]$ \quad $(1,1)$  &   $[2,2]$ \quad $(1,2,2,1)$ \\
   \includegraphics[width=65pt]{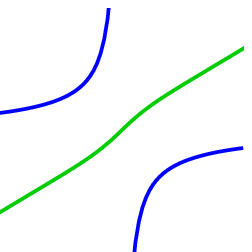}&
   \includegraphics[width=65pt]{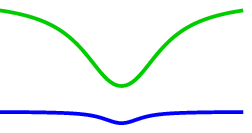}&
   \includegraphics[width=65pt]{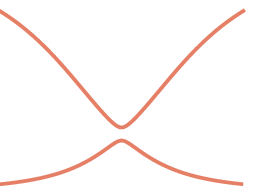}\\\hline
%%%%%%%%%%%%%%%%%%%%%%%%%%%  Oloid family
   \includegraphics[width=125pt]{pictures/unLinked}&
   \includegraphics[width=95pt]{pictures/Oloid}&\rule{0pt}{70pt}
   \includegraphics[width=100pt]{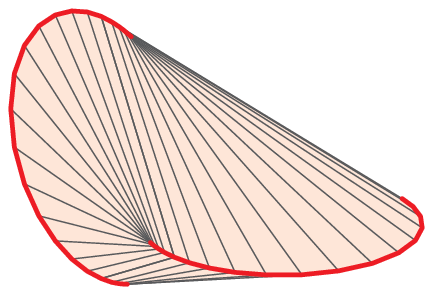}\\
  $[2,2]$ \quad $(1,1,2,2)$ & $[2,2]$ \quad $(1,2,1,2)$ & $[2,2]$ \quad $(1,S,2)$  \\
   \includegraphics[width=65pt]{pictures/unLinked_wide_curve}&
   \includegraphics[width=65pt]{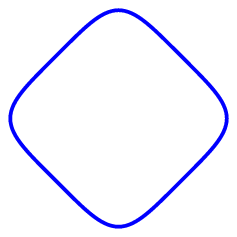}&
   \includegraphics[width=65pt]{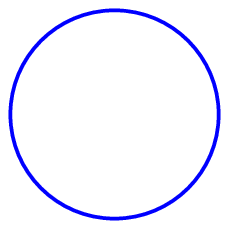}
\\\hline
  \end{tabular}
 \caption{Some convex hulls, intersection and order types, and edge curves.}
  \label{T:GalleryI}
\end{table}
%%%%%%%%%%%%%%%%%%%%%%%%%%%%%%%%%%%%%%%%%%%%%%%%%%%%%%%%%%%%%%%%%%%%%%%%%%%%%%%%%
%%%%%%%%%%%%%%%%%%%%%%%%%%%%%%%%%%%%%%%%%%%%%%%%%%%%%%%%%%%%%%%%%%%%%%%%%%%%%%%%%
\begin{table}[htb]

  \begin{tabular}{|c|c|c|}\hline
%%%%%%%%%%%%%%%%%  The first is irreducible, but singular, and the next two come from cylinders
   \includegraphics[width=80pt]{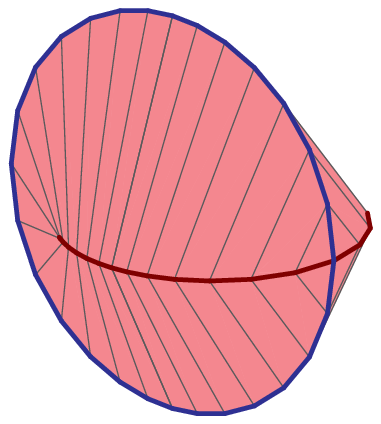}\rule{0pt}{93pt}&
   \includegraphics[width=90pt]{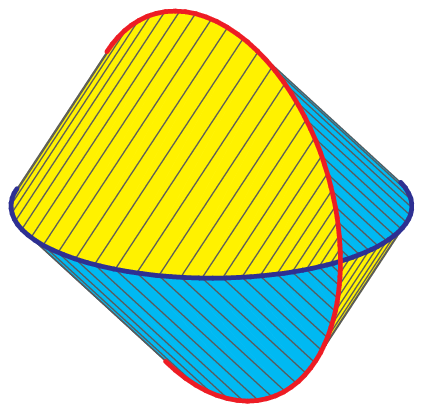}&
   \includegraphics[width=110pt]{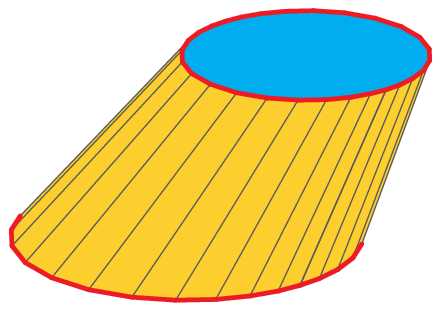}\\
   $[2,2]$ \quad $(1,2,S)$ & $[2,2]$ \quad $(S,S)$ & $[0,0]$ \quad $(2c)$ \\
   \includegraphics[width=95pt]{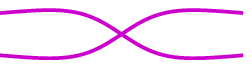}&
   \includegraphics[width=65pt]{pictures/orthoPlane_curve}&
   \includegraphics[width=65pt]{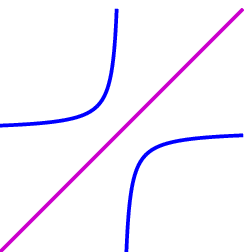}\\\hline
  \end{tabular}
%%%%%%%%%%%%%%%%%%%%%%%%%%%%%%%%%%%%%%%%%%%%%%%%%%%%%%%%%%%%%%%%%%%%%%%%%%%%%%%%%
 \caption{More convex hulls, intersection and order types, and edge curves.}
  \label{T:GalleryII}
\end{table}
%%%%%%%%%%%%%%%%%%%%%%%%%%%%%%%%%%%%%%%%%%%%%%%%%%%%%%%%%%%%%%%%%%%%%%%%%%%%%%%%%
\begin{table}[htb]
  \begin{tabular}{|c|c|c|}\hline
%%%%%%%%%%%%%%%%%  Exactly on circle is tangent to the plane of the other
   \raisebox{10pt}{\includegraphics[width=102pt]{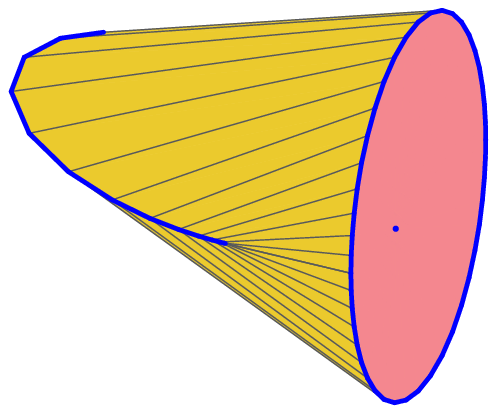}}&
   \includegraphics[width=80pt]{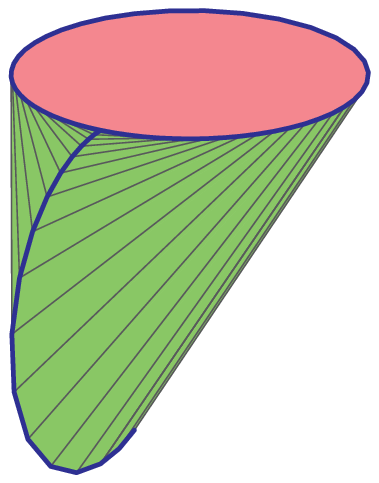}\rule{0pt}{105pt}&
   \includegraphics[width=95pt]{pictures/tangentToPlaneI}\\
    $[2,1]$ \quad  $(1,2,1)$ & $[2,1]$ \quad $(1,S)$ & $[2,1]$ \quad $(1,1,2)$ \\
   \includegraphics[width=65pt]{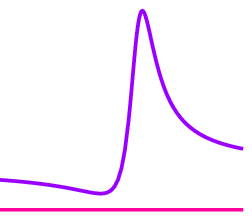}&
   \includegraphics[width=65pt]{pictures/tangentToPlaneIV_curve}&
   \includegraphics[width=70pt]{pictures/tangentToPlaneI_curve}\\\hline   
%%%%%%%%%%%%%%%%%%%%%%%%%%%%%%%%%%%%%%%%%%%%%%%%%%%%%%%%%%%%%%%%%%%%%%%%%%%%%%%
   \raisebox{10pt}{\includegraphics[width=120pt]{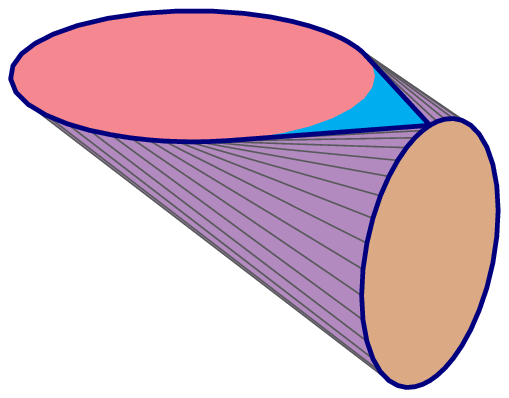}}&
   \includegraphics[width=80pt]{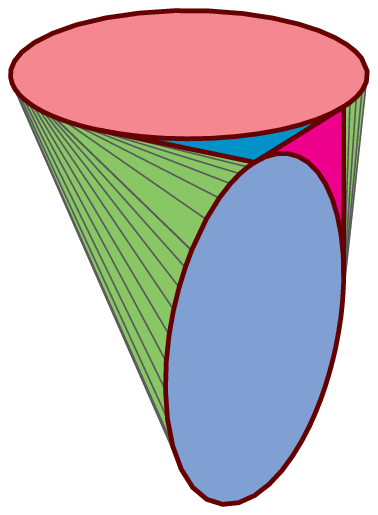}\rule{0pt}{113pt}&
   \raisebox{10pt}{\includegraphics[width=95pt]{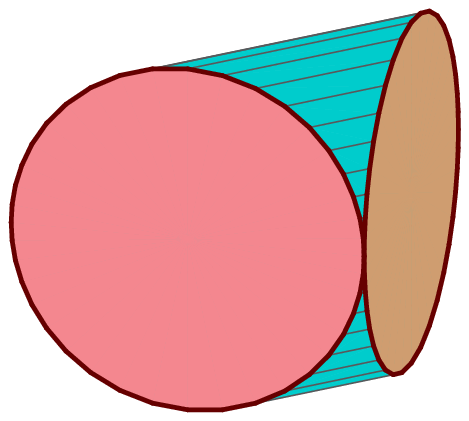}}\\
   $[1,0]$ \quad  $(1)$& $[1,1]$ \quad $(1,2)$ &$[1,1]$ \quad $(S)$\\
   \includegraphics[width=65pt]{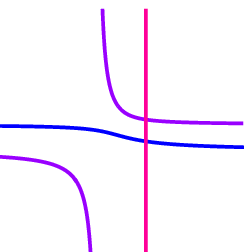}&
   \includegraphics[width=65pt]{pictures/tangentToPlaneIII_curve}&
   \includegraphics[width=65pt]{pictures/Hinged_curve}\\\hline
  \end{tabular}
  \caption{More convex hulls, intersection and order types, and edge curves.}
  \label{T:GalleryIII}
\end{table}
%%%%%%%%%%%%%%%%%%%%%%%%%%%%%%%%%%%%%%%%%%%%%%%%%%%%%%%%%%%%%%%%%%%%%%%%%%%%%%%%%

%%%%%%%%%%%%%%%%%%%%%%%%%%%%%%%%%%%%%%%%%%%%%%%%%%%%%%%%%%%%%%%%%%%%%%%%%%%%%%%%%

The order type of the circles determines the
combinatorial type of the face lattice of their convex hull \defcolor{$K$}.
Describing the face lattice means identifying all (families of) faces of $K$, their incidence relations, and which
are exposed/not exposed.
Throughout, \demph{$D_i$} is the disc of the circle $C_i$.
We invite the reader to peruse our gallery in Tables~\ref{T:GalleryI},~\ref{T:GalleryII}, and~\ref{T:GalleryIII}
while reading this classification. 
Our main result is the following.

%%%%%%%%%%%%%%%%%%%%%%%%%%%%%%%%%%%%%%%%%%%%%%%%%%%%%%%%%%%%%%%%%%%%%%%%%%%%%%%%
\begin{theorem}\label{Th:FaceLattice}
 The order type of $C_1,C_2$ determines the combinatorial type of the face lattice of $K$, as summarized in
 Table~\ref{T:faceLattices}.
 There are eleven distinct combinatorial types of face lattice.
 The combinatorial type of the face lattice, together with the real algebraic geometry of its edge curve,
 determines the order type.
\end{theorem}
%%%%%%%%%%%%%%%%%%%%%%%%%%%%%%%%%%%%%%%%%%%%%%%%%%%%%%%%%%%%%%%%%%%%%%%%%%%%%%%%

We determine the face lattice for each order type.
Some general statements are given in preliminary results which precede our proof of
Theorem~\ref{Th:FaceLattice}. 
The statements are asymmetric, with the symmetric statement obtained by interchanging 1 and 2.
We first study the section $\defcolor{\kappa_1}:=K\cap\Pi_1$ of $K$,  which contains $D_1$.

%%%%%%%%%%%%%%%%%%%%%%%%%%%%%%%%%%%%%%%%%%%%%%%%%%%%%%%%%%%%%%%%%%%%%%%%%%%%%%%%
\begin{lemma}\label{L:section}
 $\kappa_1=\conv(C_1,C_2\cap\Pi_1)$.
\end{lemma}
%%%%%%%%%%%%%%%%%%%%%%%%%%%%%%%%%%%%%%%%%%%%%%%%%%%%%%%%%%%%%%%%%%%%%%%%%%%%%%%%%

%%%%%%%%%%%%%%%%%%%%%%%%%%%%%%%%%%%%%%%%%%%%%%%%%%%%%%%%%%%%%%%%%%%%%%%%%%%%%%%%%
\begin{proof}
 As $D_i=\conv(C_i)$, $K=\conv(D_1,D_2)$.
 Therefore a point $x\in K$
 is a convex combination $\lambda y+\mu z$ ($\lambda,\mu\geq 0$ with $\lambda+\mu=1$) of points $y\in D_1$ and 
 $z\in D_2$. 
 If  $x\in\kappa_1\subset\Pi_1$, then as 
 $y\in\Pi_1$, we must have that $z\in D_2\cap\Pi_1= \conv(C_2\cap\Pi_1)$.
\end{proof}
%%%%%%%%%%%%%%%%%%%%%%%%%%%%%%%%%%%%%%%%%%%%%%%%%%%%%%%%%%%%%%%%%%%%%%%%%%%%%%%%%

%%%%%%%%%%%%%%%%%%%%%%%%%%%%%%%%%%%%%%%%%%%%%%%%%%%%%%%%%%%%%%%%%%%%%%%%%%%%%%%%%
\begin{corollary}\label{C:2face}
 If $C_2\cap\Pi_1\subset D_1$, then we have $\kappa_1=D_1$.
 Otherwise, $\kappa_1$ is the convex hull of $D_1$ and the one or two points of $C_2\cap\Pi_1$
 exterior to $D_1$.
 A point $p\in C_1$ is an extreme point of $\kappa_1$ if and only if $C_1$ and $C_2\cap\Pi_1$ lie on the same side
 of $T_pC_1$. 
 An extreme point $p\in C_1$ of $\kappa_1$ is not exposed if and only if $T_pC_1$ meets $C_2\smallsetminus\{p\}$.
 Extreme points of $\kappa_1$ are extreme points of $K$ and nonexposed points of $\kappa_1$ are nonexposed in
 $K$.
 Finally, $\kappa_1$ is a face of $K$ if and only if $m_2\leq 1$.
\end{corollary}
%%%%%%%%%%%%%%%%%%%%%%%%%%%%%%%%%%%%%%%%%%%%%%%%%%%%%%%%%%%%%%%%%%%%%%%%%%%%%%%%%

%%%%%%%%%%%%%%%%%%%%%%%%%%%%%%%%%%%%%%%%%%%%%%%%%%%%%%%%%%%%%%%%%%%%%%%%%%%%%%%%%
\begin{proof}
 The first two statements follow from Lemma~\ref{L:section}.
 The next two about extreme points $p$ of $\kappa_1$ follow as $T_pC_1$ is the only possible supporting line to
 $\kappa_1$ at $p$.
 The next, about extreme points of $K$ and its section $\kappa_1$, follows by Lemma~\ref{L:section}, and the last is
 immediate as $C_2$ lies on one side of $\Pi_1$ if and only if $|C_2\cap\Pi_1|<2$.
\end{proof}
%%%%%%%%%%%%%%%%%%%%%%%%%%%%%%%%%%%%%%%%%%%%%%%%%%%%%%%%%%%%%%%%%%%%%%%%%%%%%%%%%

By Lemma~\ref{L:stBis}, a general point $p\in C_1$ lies on two stationary bisecants.
If $p\in K$ is extreme, then these may support one-dimensional \demph{bisecant faces} of $K$.
We determine the bisecant faces meeting most extreme points.
Any plane supporting an extreme point $p\in C_1$ contains $T_p C_1$.
If such a plane does not meet $C_2$, then $p$ is exposed.

%%%%%%%%%%%%%%%%%%%%%%%%%%%%%%%%%%%%%%%%%%%%%%%%%%%%%%%%%%%%%%%%%%%%%%%%%%%%%%%%%
\begin{lemma}\label{L:bisecantFaces}
 Let $p\in K$ be an extreme point of $K$.
 If $T_p C_1$ neither meets $C_2$ nor lies in $\Pi_2$, then $p$ is exposed.
 Such a point $p$ lies on one bisecant face if $m_2\leq 1$ and two if $m_2=2$.
 When there are two, one is on each side of $\Pi_1$.
\end{lemma}
%%%%%%%%%%%%%%%%%%%%%%%%%%%%%%%%%%%%%%%%%%%%%%%%%%%%%%%%%%%%%%%%%%%%%%%%%%%%%%%%%

%%%%%%%%%%%%%%%%%%%%%%%%%%%%%%%%%%%%%%%%%%%%%%%%%%%%%%%%%%%%%%%%%%%%%%%%%%%%%%%%%
\begin{proof}
 Let $p\in C_1$ be an extreme point of $K$ such that $T_p C_1$ neither meets $C_2$ nor lies in $\Pi_2$.
 By Corollary~\ref{C:2face}, $C_1$ and $C_2\cap\Pi_1$ lie on the same side of $T_pC_1$ in $\Pi_1$.
 In the pencil $\RR\PP^1$ of planes containing $T_pC_1$, those meeting $K$ form an interval $I$ containing 
 $\Pi_1$ and an interval $\gamma$ of planes meeting $C_2$.
 Each endpoint of $\gamma$ is a plane containing a stationary bisecant through $p$.
 Our assumptions on $p$ and $T_pC_1$ imply that $I\neq\RR\PP^1$, so that $p$ is exposed. 
 If $m_2\leq 1$, then $\Pi_1$ is one endpoint of $I$ and the other is an endpoint of $\gamma$, otherwise the
 endpoints of $I$ are the endpoints of $\gamma$ and $\Pi_1$ is an interior point, which proves the lemma.
\end{proof}
%%%%%%%%%%%%%%%%%%%%%%%%%%%%%%%%%%%%%%%%%%%%%%%%%%%%%%%%%%%%%%%%%%%%%%%%%%%%%%%%%

%%%%%%%%%%%%%%%%%%%%%%%%%%%%%%%%%%%%%%%%%%%%%%%%%%%%%%%%%%%%%%%%%%%%%%%%%%%%%%%%%
\begin{remark}\label{R:program}
 Corollary~\ref{C:2face} identifies the 2-faces, extreme points, and some nonexposed
 points of $K$. 
 Lemma~\ref{L:bisecantFaces} identifies most exposed points and bisecant edges.
 The rest of the face lattice is determined in the proof of Theorem~\ref{Th:FaceLattice}.
 We first understand the boundary of each section $\kappa_i=K\cap\Pi_i$.
 Fig.~\ref{F:kappas} shows the possibilities when $\kappa_i$ is not the disc $D_i$.
%%%%%%%%%%%%%%%%%%%%%%%%%%%%%%%%%%%%%%%%%%%%%%%%%%%%%%%%%%%%%%%%%%%%%%%%%%%%%%%%%
\begin{figure}[htb]
 \begin{picture}(84,69)(0,-17)
   \put(0,0){\includegraphics{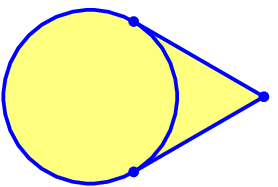}}
   \put(21,23){$D_i$}  \put(37,53){$p_1$} \put(37,-3){$p_2$}  \put(79,26){$q_1$}
   \put(35,-15){(a)}
 \end{picture}
 \qquad
 \begin{picture}(82,74)(0,-12)
   \put(0,0){\includegraphics{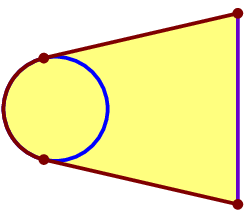}}
    \put(11,27.5){$D_i$}  \put(8,50){$p_1$}\put(8,7){$p_2$} 
    \put(72, 0){$q_1$}    \put(72,57){$q_2$}
  \put(35,-10){(b)}
  \end{picture}
 \qquad
 \begin{picture}(104,76)(-7,-12)
  \put(0,5){\includegraphics{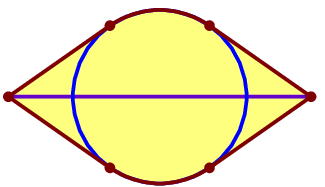}}
  \put(-8,30){$q_1$}  \put(93,30){$q_2$}
  \put(25,57){$p_4$}  \put(60,57){$p_3$}
  \put(25, 4){$p_1$}  \put(60, 4){$p_2$}
  \put(41,17){$D_i$}
  \put(41,-10){(c)}
 \end{picture}
\caption{Some possible slices $\kappa_i$.}
\label{F:kappas}
\end{figure}
%%%%%%%%%%%%%%%%%%%%%%%%%%%%%%%%%%%%%%%%%%%%%%%%%%%%%%%%%%%%%%%%%%%%%%%%%%%%%%%%%
There, $q_1$ and $q_2$ are points of the other circle on the boundary of $\kappa_i$ and 
points $p_j$ are nonexposed points of $C_i$ as $T_{p_j}C_i$ meets $q_1$ or $q_2$.
The line segment between $q_1$ and $q_2$ is where the disc of the second circle meets $\Pi_i$.
\end{remark}
%%%%%%%%%%%%%%%%%%%%%%%%%%%%%%%%%%%%%%%%%%%%%%%%%%%%%%%%%%%%%%%%%%%%%%%%%%%%%%%%%

%%%%%%%%%%%%%%%%%%%%%%%%%%%%%%%%%%%%%%%%%%%%%%%%%%%%%%%%%%%%%%%%%%%%%%%%%%%%%%%%%
\begin{proof}[of Theorem~\ref{Th:FaceLattice}]
We give separate arguments for each order type.

\noindent{\bf Order type $\emptyset$.}
 By Corollary~\ref{C:2face}, both discs are faces of $K$, and every point of the circles is extreme.
 By Lemma~\ref{L:bisecantFaces}, all points of the circles are exposed, and each point lies on exactly one bisecant
 face. 

 The same description holds for order type $(2c)$.
 As the edge curve for order type $\emptyset$ is smooth and of genus 1, while that for order type
 $(2c)$ is singular, the edge curve distinguishes these two order types.

\noindent{\bf Order type $(1,1)$.}
 Since $m_1=2$ and $m_2=0$, $D_1$ is the only 2-face.
 The section $\kappa_2$ is similar to Fig.~\ref{F:kappas} (b), so the extreme points on $C_2$ form an arc
 $\arc{p_1,p_2}$ whose endpoints are not exposed, each lying on one bisecant
 edge. 
 The interior points of $\arc{p_1,p_2}$ are exposed by Lemma~\ref{L:bisecantFaces} and each lies on two bisecant
 edges.
 Similarly, every point of $C_1$ is exposed and lies on one bisecant edge.

 The same description holds for order type $(1,2,1)$.
 Its edge curve is singular, while order type $(1,1)$
 has a smooth edge curve.

\noindent{\bf Order type $(1,2,2,1)$.}
 Since  $m_1=m_2=2$, $K$ has no 2-faces.
 Since $C_2$ meets the interior of $D_1$, Corollary~\ref{C:2face} implies that every point of $C_1$ is extreme and 
 $C_2$ has two intervals of extreme points.  
 The four endpoints are not exposed and each lies on one  bisecant edge.
 By Lemma~\ref{L:bisecantFaces}, every point of $C_1$ and of the interior of the arcs on $C_2$ is exposed and  lies
 on two bisecant edges. 

\noindent{\bf Order type $(1,1,2,2)$.}
 By Corollary~\ref{C:2face}, $K$ has no 2-faces and each circle has one arc of extreme points, as the sections
 $\kappa_i$ are similar to Fig.~\ref{F:kappas} (a).
 As before, each endpoint of an arc is not exposed and lies on one bisecant edge, and each interior point
 of an arc is exposed and lies on two bisecant edges.

 The same description holds for order types $(1,2,1,2)$ and $(1,S,2)$.
 The edge curve in type $(1,1,2,2)$ has two real components as seen in Example~\ref{Ex:edge_curves2}, while for 
 type $(1,2,1,2)$ there is one real component, and both are smooth.
 For type $(1,S,2)$, the edge curve is singular at the shared point.

\noindent{\bf Order type $(1,2,S)$.}
 Again, $K$ has no 2-faces.
 All points of $C_1$ are extreme and $C_2$ has an arc of extreme points whose
 endpoints are not exposed and each lies on one bisecant edge.
 Also, all interior points of that arc and of $C_1$---except possibly the shared point
 $p$---are exposed and lie on two bisecant edges.
 The tangents $T_pC_1$ and $T_pC_2$ span a plane exposing $p$ and $p$ lies on no bisecant edges.

\noindent{\bf Order type $(S,S)$.}
 There are no 2-faces and as in type $(1,2,S)$ every point of the circles is extreme, and the nonshared
 points are exposed and each lies on two bisecant edges.
 Each shared point is exposed by the plane spanned by the two tangents at that point and neither shared point lies
 on a bisecant edge.

\noindent{\bf Order type $(1,S)$.}
 The only 2-face is $D_1$.
 Every point of $C_1$ is extreme and $C_2$ has an arc $\arc{p_1,p_2}$ of extreme points with one endpoint,
 say $p_1$, the shared point where $C_2$ is tangent to $\Pi_1$.
 Neither endpoint is exposed and $p_2$ lies on one bisecant edge (the bisecant $T_{p_1}C_2$ meets the interior of
 $D_1$).
 By Lemma~\ref{L:bisecantFaces}, every point of $C_1$ except $p_1$ lies on one bisecant edge and every interior
 point of  $\arc{p_1,p_2}$ lies on two bisecant edges, and all of these are exposed.

\noindent{\bf Order type $(1,1,2)$.}
 The only 2-face is $\kappa_1$ and its shape is as in Fig~\ref{F:kappas} (a) with the
 vertex $q_1$ where $D_2$ is tangent to $\Pi_1$.
 There is an arc $\arc{p_1,p_2}$ of extreme points of $C_1$ whose endpoints are not exposed with each lying on a
 bisecant edge $\overline{p_i,q_1}$.
 The section $\kappa_2$ has the same shape and $C_2$ has an arc $\arc{q_1,q_2}$ of extreme points with neither
 endpoint exposed.
 The point $q_2$ lies on one bisecant edge along $T_{q_2}C_2$ and $q_1$ lies on two bisecant edges
 $\overline{p_i,q_1}$. 
 Neither of the edges $\overline{p_i,q_1}$ is exposed as $\Pi_1$ is the only supporting plane of $K$ containing
 either edge.
 Finally, by Lemma~\ref{L:bisecantFaces}, interior points of the arcs are exposed, with those from
 $\arc{p_1,p_2}$ lying on one bisecant edge and those from $\arc{q_1,q_2}$ lying on two.\smallskip

 In the order types of the last row of Table~\ref{T:GalleryIII}, the circle $C_2$ is tangent to $\Pi_1$ at a
 point $q_1$ and the tangent $T_{q_1}C_2$ does not meet the interior of $D_1$.
 In the pencil of planes containing $T_{q_1}C_2$, $\Pi_1$ and $\Pi_2$ are
 the endpoints of an interval of planes meeting $K\smallsetminus T_{q_1}C_2$ and of an interval of planes that
 meet $K$ only in $T_{q_1}C_2\cap K$.
 Thus, both sections $\kappa_1$ and $\kappa_2$ are 2-faces of $K$ and the face $T_{q_1}C_2\cap K$ is
 exposed.\smallskip 

\noindent{\bf Order type $(1)$.}
 Here, $m_1=1$ and $m_2=0$.
 The 2-face $\kappa_2$ has the same shape as in order type $(1,1,2)$.
 The description of the points and bisecant edges meeting $C_2$ is also the same.
 By Lemma~\ref{L:bisecantFaces} and the preceding observation, every point of $C_1$ is exposed, and all lie on a
 unique bisecant edge except $q_1$, which lies on the two nonexposed bisecant edges $\overline{p_i,q_1}$.

\noindent{\bf Order type $(1,2)$.}
 This is the most complicated.
 Each circle is tangent to the plane of the other, sharing a tangent line, and the
 description is symmetric 
 in the indices $1$ and $2$.
 The 2-faces are the sections $\kappa_1$ and $\kappa_2$, with the description for each is nearly the same as for
 $\kappa_1$ in order type $(1,1,2)$.
 The exception is the bisecant edge $\overline{p_1,q_1}$ lying along the shared tangent.
 This is exposed, but neither endpoint is exposed.
 It is also isolated from the other bisecant edges, which form a continuous family.

\noindent{\bf Order type $(S)$.}
 The two circles are mutually tangent at a point $p$.
 The 2-faces are $D_1$ and $D_2$, every point of either circle is extreme, including $p$, and each (except for $p$)
 lies on one bisecant edge.
\end{proof}
%%%%%%%%%%%%%%%%%%%%%%%%%%%%%%%%%%%%%%%%%%%%%%%%%%%%%%%%%%%%%%%%%%%%%%%%%%%%%%%%%

Table~\ref{T:faceLattices} summarizes the face lattices by order type.
In it, when $m_i=1$, $p_i$ is the point were $C_i$ is tangent to the plane of the other circle.\smallskip
%%%%%%%%%%%%%%%%%%%%%%%%%%%%%%%%%%%%%%%%%%%%%%%%%%%%%%%%%%%%%%%%%%%%%%%%%%%%%%%%%
%Ata.tex
%
%   This is the table for face lattices
%
\begin{table}[htb]
  \begin{tabular}{
		p{\dimexpr.15\linewidth-2\tabcolsep-1.3333\arrayrulewidth}% column 1
		p{\dimexpr.28\linewidth-2\tabcolsep-1.3333\arrayrulewidth}% column 2
		p{\dimexpr.40\linewidth-2\tabcolsep-1.3333\arrayrulewidth}% column 3
		p{\dimexpr.16\linewidth-2\tabcolsep-1.3333\arrayrulewidth}% column 4
	}
	
		\centering \textbf{Order Type}&
		\centering \textbf{0-faces}& 
		\centering\textbf{1-faces}&
		\centering\arraybackslash \textbf{2-faces}\\ \midrule
		
		%%%%%%% TYPE [0,0] %%%%%%%%%
		\centering $\emptyset$&
		\centering Points on $C_1\cup C_2$& 
		\centering One family parameterized by $C_1$&
		\centering \arraybackslash$D_1$, $D_2$ \\ \midrule
		
		%%%%%%% TYPE [2,0](1,1) %%%%%%
		\centering\multirow{2}{*}{$(1,1)$}&
		\centering Points on $C_1$ and points on an arc of $C_2$&
		\centering\multirow{2}{*}{One family parameterized by $C_1$}&
		\centering\arraybackslash\multirow{2}{*}{$D_1$}\\ \midrule
		
		%%%%%%% TYPE [2,2](1,2,2,1) %%%% 
		\centering\multirow{2}{*}{$(1,2,2,1)$}&
		\centering Points on $C_1$ and points on two arcs of $C_2$&
		\centering\multirow{2}{*}{Two families parameterized by $C_1$}&
		\centering\arraybackslash\multirow{3}{*}{None}     \\ \midrule
		
		%%%%%%% TYPE [2,2](1,1,2,2) %%%% 
		\centering\multirow{2}{*}{$(1,1,2,2)$}&
		\centering Points on an arc of $C_1$\\ and an arc of $C_2$& 
		\centering One family parameterized by a 2-fold branched cover of an arc&
		\centering\arraybackslash\multirow{2}{*}{None}     \\ \midrule
		
		%%%%%%% TYPE [2,2](1,2,1,2) %%%% 
		
		\centering$(1,2,1,2)$&
		\multicolumn{3}{c}{  Same as order type $(1,1,2,2)$}\\ \midrule
		
		%%%%%%% TYPE [2,2](1,S,2) %%%% 
		
		\centering$(1,S,2)$&
		\multicolumn{3}{c}{Same as order type $(1,1,2,2)$}\\ \midrule
		
		%%%%%%% TYPE [2,2](1,2,S) %%%% 
		\centering\multirow{2}{*}{$(1,2,S)$}&
		\centering Points on $C_1$ and \\ an arc of $C_2$& 
		\centering Two families parameterized\\ by $C_1\smallsetminus C_2$&
		\centering\arraybackslash\multirow{2}{*}{None}     \\ \midrule
		
		%%%%%%% TYPE [2,2](S,S) %%%% 
		\centering\multirow{2}{*}{$(S,S)$}&
		\centering \multirow{2}{*}{Points on $C_1 \cup C_2$} & 
		\centering Four families with two parameterized by each arc $C_1\smallsetminus C_2$&
		\centering\arraybackslash\multirow{2}{*}{None}     \\ \midrule
		
		%%%%%%% TYPE [0,0](2c) %%%% 
		\centering$(2c)$&
		\multicolumn{3}{c}{Same as order type $\emptyset$}\\ \midrule
		
		%%%%%%% TYPE [2,1](1,S) %%%% 
		\centering\multirow{2}{*}{$(1,S)$}&
		\centering Points on $C_1$ and\\ an arc of $C_2$ & 
		\centering One family parameterized\\ by $C_1\smallsetminus C_2$&
		\centering\arraybackslash\multirow{2}{*}{$D_1$}     \\ \midrule
		
		%%%%%%% TYPE [2,1](1,1,2) %%%% 
		\centering\multirow{2}{*}{$(1,1,2)$}&
		\centering Points on an arc of $C_1$\\ and an arc of $C_2$ & 
		\centering One family parameterized by\\ the arc of $C_1$&
		\centering\arraybackslash\multirow{2}{*}{$\conv(D_1,p_2)$}     \\ \midrule
		
		%%%%%%% TYPE [2,1](1,2,1) %%%% 
		\centering$(1,2,1)$&
		\multicolumn{3}{c}{Same as order type $(1,1)$}\\ \midrule
		
		%%%%%%% TYPE [1,0](1) %%%% 
		\centering\multirow{2}{*}{$(1)$}&
		\centering Points on $C_1$ and\\ an arc of $C_2$ & 
		\centering One family parameterized by\\ the arc on $C_2$&
		\centering\arraybackslash$D_1$, $\conv(D_2,p_1)$    \\ \midrule
		
		%%%%%%% TYPE [1,1](1,2) %%%% 
		\centering\multirow{2}{*}{$(1,2)$}&
		\centering Points on an arc of $C_1$\\ and an arc of $C_2$ & 
		\centering One family parameterized by either arc, 
                 and an isolated bisecant $\overline{p_1,p_2}$&
		\centering\arraybackslash$\conv(D_1,p_2)$, $\conv(D_2,p_1)$ \\ \bottomrule
		
		%%%%%%% TYPE [1,1](S) %%%% 
		\centering\multirow{2}{*}{$(S)$}&
		\centering \multirow{2}{*}{Points on $C_1 \cup C_2$} & 
		\centering One family parameterized by either circle except the common point&
		\centering\arraybackslash\multirow{2}{*}{$D_1$, $D_2$}     \\ \midrule
	\end{tabular}
 \caption{Face lattices.}\label{T:faceLattices}
	
\end{table}
%%%%%%%%%%%%%%%%%%%%%%%%%%%%%%%%%%%%%%%%%%%%%%%%%%%%%%%%%%%%%%%%%%%%%%%%%%%%%%%%%

%%%%%%%%%%%%%%%%%%%%%%%%%%%%%%%%%%%%%%%%%%%%%%%%%%%%%%%%%%%%%%%%%%%%%%%%%%%%%%%%%

By Theorem~\ref{Th:one}, the convex hull $K$ is a spectrahedral shadow.
We use our classification to describe when $K$ is a spectrahedron.

%%%%%%%%%%%%%%%%%%%%%%%%%%%%%%%%%%%%%%%%%%%%%%%%%%%%%%%%%%%%%%%%%%%%%%%%%%%%%%%%%
\begin{lemma}\label{QuadricLemma}
 Let $C_1,C_2 \in \PP^3$ be conics in distinct planes $\Pi_1$ and $\Pi_2$. 
 If $C_1 \cap \Pi_2 = C_2\cap \Pi_1$, then $C_1$ and $C_2$ lie on a pencil of quadrics.
\end{lemma}
%%%%%%%%%%%%%%%%%%%%%%%%%%%%%%%%%%%%%%%%%%%%%%%%%%%%%%%%%%%%%%%%%%%%%%%%%%%%%%%%%

%%%%%%%%%%%%%%%%%%%%%%%%%%%%%%%%%%%%%%%%%%%%%%%%%%%%%%%%%%%%%%%%%%%%%%%%%%%%%%%%%
\begin{proof}
 Since $C_1,C_2$ lie on the singular quadric $\Pi_1\cup\Pi_2$, we need only find a second quadric containing
 them. 
 Choose coordinates $[x:y:z:w]$ for $\PP^3$ so that $\Pi_1$ is defined by $w=0$ and $\Pi_2$ by $z=0$.
 Then $C_1$ and $C_2$ are given by homogeneous quadratic polynomials $f(x,y,z)=0$ and $g(x,y,w)=0$.
 Since $C_1 \cap \Pi_2 = \Pi_1\cap C_2$, the forms $f(x,y,0)$ and $g(x,y,0)$ define the same scheme, so they are
 proportional.
 Scaling $g$ if necessary, we may assume that $f(x,y,0)=g(x,y,0)$.
 Define $\defcolor{h(x,y,z,w)}$ to be $f(x,y,z)+g(x,y,w)-f(x,y,0)$.
 It follows that  $h(x,y,z,0) = f(x,y,z)$ and $h(x,y,0,w) = g(x,y,w)$, and
 thus $C_1$ and $C_2$ lie on the quadric defined by $h$.  
\end{proof}
%%%%%%%%%%%%%%%%%%%%%%%%%%%%%%%%%%%%%%%%%%%%%%%%%%%%%%%%%%%%%%%%%%%%%%%%%%%%%%%%%

%%%%%%%%%%%%%%%%%%%%%%%%%%%%%%%%%%%%%%%%%%%%%%%%%%%%%%%%%%%%%%%%%%%%%%%%%%%%%%%%%
\begin{theorem}\label{Th:isSectrhedron}
 The convex hull of two circles $C_1$ and $C_2$ lying in distinct planes in $\RR^3$
 is a spectrahedron only if they have order type $(SS)$ or $(2c)$ or $(S)$.
\end{theorem}
%%%%%%%%%%%%%%%%%%%%%%%%%%%%%%%%%%%%%%%%%%%%%%%%%%%%%%%%%%%%%%%%%%%%%%%%%%%%%%%%%

%%%%%%%%%%%%%%%%%%%%%%%%%%%%%%%%%%%%%%%%%%%%%%%%%%%%%%%%%%%%%%%%%%%%%%%%%%%%%%%%%
\begin{proof} 
 We have that $C_1 \cap \Pi_2 = C_2\cap \Pi_1$ in $\PP^3$ if and only if 
 the circles have order type $(SS)$ or $(2c)$ or $(S)$. 
 By Lemma~\ref{QuadricLemma}, $C_1$ and $C_2$ lie on a pencil $Q_1+ t Q_2$ of quadrics. 
 Following Example 2.3 in~\cite{RS12}, this pencil of quadrics contains singular quadrics given by the real roots
 of $\det(Q_1 + tQ_2)$. 
 Such a singular quadric is given by the determinant of a $2 \times 2$ matrix polynomial $Ax + By + Cz + D$, 
 and the block diagonal matrix with blocks $A$, $B$, $C$, and $D$ represents $\conv(C)$ as a spectrahedron.

 By Corollary~\ref{C:2face}, $K$ has a nonexposed face when a tangent line to one circle meets the other
 circle in a different point.
 This occurs for all the remaining order types of the circles $C_1$ and $C_2$, except type 
 $\emptyset$ where $C_1\cap C_2=\emptyset$ in $\PP^3$. 
 In this case, the edge curve is irreducible with two connected real components and the
 edge surface meets the interior of $\conv(C)$  (as there are internal stationary bisecants). 
 Thus, $\conv(C)$ is not a basic semialgebraic set and thus not a spectrahedron.
\end{proof}
%%%%%%%%%%%%%%%%%%%%%%%%%%%%%%%%%%%%%%%%%%%%%%%%%%%%%%%%%%%%%%%%%%%%%%%%%%%%%%%%%

%%%%%%%%%%%%%%%%%%%%%%%%%%%%%%%%%%%%%%%%%%%%%%%%%%%%%%%%%%%%%%%%%%%%%%%%%%%%%%%%%
%
\section{Convex hulls through duality}\label{S:Duality}

We sketch an alternative approach to studying the convex hull $K$ of two circles that uses projective
duality. 
This is inspired by the paper~\cite{ellipsoids} and accompanying video~\cite{ellipsoids_video} that explains
a solution to the problem of determining the convex hull of three ellipsoids in $\RR^3$.

%%%%%%%%%%%%%%%%%%%%%%%%%%%%%%%%%%%%%%%%%%%%%%%%%%%%%%%%%%%%%%%%%%%%%%%%%%%%%%%%%
%
Points \defcolor{$\check{\Pi}$} of the dual projective space \defcolor{$\check{\PP}^3$} correspond to planes
$\Pi$ of the primal space $\PP^3$.
A line \defcolor{$\check{\ell}$} represents the pencil of planes
containing a fixed line $\ell\subset\PP^3$, and a plane \defcolor{$\check{o}$} represents the net of
planes incident on a point $o\in\PP^3$.
The dual $\defcolor{\check{C}}\subset\check{\PP}^3$ of a conic $C\subset\PP^3$ is the set of planes that
contain a line tangent to $C$.

%%%%%%%%%%%%%%%%%%%%%%%%%%%%%%%%%%%%%%%%%%%%%%%%%%%%%%%%%%%%%%%%%%%%%%%%%%%%%%%%%
\begin{lemma}
 The dual $\check{C}$ to a conic $C$ is a quadratic cone in $\check{\PP}^3$ with vertex $\check{\Pi}$
 corresponding to the plane $\Pi$ of $C$.
\end{lemma}
%%%%%%%%%%%%%%%%%%%%%%%%%%%%%%%%%%%%%%%%%%%%%%%%%%%%%%%%%%%%%%%%%%%%%%%%%%%%%%%%%

%%%%%%%%%%%%%%%%%%%%%%%%%%%%%%%%%%%%%%%%%%%%%%%%%%%%%%%%%%%%%%%%%%%%%%%%%%%%%%%%%
\begin{proof}
 The pencil of planes containing the tangent line $T_pC$ to $C$ is a line lying on
 $\check{C}$ that meets $\check{\Pi}$ as $T_pC\subset\Pi$.
 Thus, $\check{C}$ is a cone in $\check{\PP}^3$ with vertex $\check{\Pi}$.
 Let $o\in\PP^3$ be any point that is not on $\Pi$.
 Then the curve $\check{o}\cap\check{C}$ is the set of planes through $o$ that contain a tangent line $T_pC$ to $C$.
 As there are two such planes that contain a general line $\ell$ through $o$---$\ell$ meets two tangents
 to $C$---the curve $\check{o}\cap\check{C}$ is a conic in $\check{o}$ and 
 $\check{C}$ is the cone over that conic with vertex $\check{\Pi}$. 
\end{proof}
%%%%%%%%%%%%%%%%%%%%%%%%%%%%%%%%%%%%%%%%%%%%%%%%%%%%%%%%%%%%%%%%%%%%%%%%%%%%%%%%%

Let $C_1,C_2$ be circles in $\RR^3\subset\RR\PP^3$ lying in distinct planes $\Pi_1,\Pi_2$ and let
$\defcolor{K}$ be the convex hull of $C_1\cup C_2$.
Let \defcolor{$o$} be any point in the interior of $K$.
We will consider the hyperplane $\check{o}\subset\check{\RR\PP}^3$ to be the hyperplane at infinity and 
set $\defcolor{\check{\RR}^3}:=\check{\RR\PP}^3\smallsetminus \check{o}$.
This is an affine space that contains every hyperplane supporting $K$ as well as all those disjoint from $K$, 
as every hyperplane incident on $o$ meets the interior of $K$.
It also contains the point \defcolor{$\check{\infty}$} corresponding to the hyperplane at infinity in $\RR\PP^3$.

For $i=1,2$, let \defcolor{$\check{C}_i$} be the cone in $\check{\RR}^3$ dual to the conic $C_i$.
If $o\in\Pi_i$, then the vertex $\check{\Pi}_i$ of $\check{C}_i$ lies at infinity ($\check{\Pi}_i\in\check{o}$)
and $\check{C}_i$ is a cylinder.
Neither dual cone contains the point $\check{\infty}$.
Let \defcolor{$\check{K}$} be the closure of the component of 
$\check{\RR}^3\smallsetminus\check{C}_1\smallsetminus\check{C}_2$ containing $\check{\infty}$.

%%%%%%%%%%%%%%%%%%%%%%%%%%%%%%%%%%%%%%%%%%%%%%%%%%%%%%%%%%%%%%%%%%%%%%%%%%%%%%%%%
\begin{proposition}\label{P:convex-dual}
 Points $\check{\Pi}$ in the interior of $\check{K}$ are exactly those whose corresponding hyperplane $\Pi$
 is disjoint from $K$.
 Points of the boundary of $\check{K}$ correspond to supporting hyperplanes of $K$, and $\check{K}$ is convex and
 bounded. 
\end{proposition}
%%%%%%%%%%%%%%%%%%%%%%%%%%%%%%%%%%%%%%%%%%%%%%%%%%%%%%%%%%%%%%%%%%%%%%%%%%%%%%%%%

We present an elementary proof of this standard result about convex sets in $\RR^d$.

%%%%%%%%%%%%%%%%%%%%%%%%%%%%%%%%%%%%%%%%%%%%%%%%%%%%%%%%%%%%%%%%%%%%%%%%%%%%%%%%%
\begin{proof}
 Choose coordinates $(x,y,z)$ for $\RR^3$ so that $o=(0,0,0)$ is the origin.
 An affine hyperplane is defined by the vanishing of an affine form $\Lambda:=ax+by+cz+d$, whose coefficients
 $[a:b:c:d]$ give homogeneous coordinates for $\check{\RR\PP}^3$.
 In these coordinates, $\check{\infty}$ is the point $[0:0:0:1]$, $\check{o}$ has equation $d=0$, and the
 points of the affine $\check{\RR}^3$ have coordinates $[a:b:c:1]$, so that $\check{\infty}$ is the origin in
 $\check{\RR}^3$. 

 Let $v=(\alpha,\beta,\gamma)\in\RR^3\smallsetminus\{(0,0,0)\}$ and consider the linear map 
 $\defcolor{\Lambda_v}\colon\RR^3\to\RR$,
\[
   \Lambda_v(x,y,z)\ :=\ \alpha x + \beta y + \gamma z\,.
\]
 Since $\Lambda_v^{-1}(0)$ is a plane containing the origin $o$, $\Lambda_v(K)$ is a closed interval
 $[\epsilon,\delta]$ with $0$ in its interior, so that $\epsilon<0<\delta$.
 Thus, the points
\[
   \Lambda_{v,t}\ \colon\ [t\alpha:t\beta:t\gamma:1]\ \ 
   \mbox{for}\ \ -\tfrac{1}{\delta}<t<-\tfrac{1}{\epsilon}
\]
 of $\check{\RR\PP}^3$ are exactly the hyperplanes in $\RR\PP^3$  parallel to $\Lambda_v^{-1}(0)$  that are
 disjoint from $K$ as $\Lambda_{v,t}(K)\subset (0,\infty)$ for $-\tfrac{1}{\delta}<t<-\tfrac{1}{\epsilon}$.

 All other planes parallel to $\Lambda^{-1}_v(0)$ meet $K$, with $\Lambda_{v,-1/\delta}$ and $\Lambda_{v,-1/\epsilon}$
 the hyperplanes in this family that support $K$.
 These supporting hyperplanes necessarily lie on $\check{C}_1\cup \check{C}_2$.
 Hence, the interior of $\check{K}$ is exactly the set of all hyperplanes disjoint from $K$ and its boundary is
 exactly the set of hyperplanes supporting $K$.

 As $o$ lies in the interior of $K$, there is a closed ball centered at $o$ of radius $1/\rho$ contained in the
 interior of $K$.
 For any unit vector $v$, the numbers $\epsilon,\delta$ defined by $\Lambda_v(K)=[\epsilon,\delta]$ satisfy
 $|1/\epsilon|, |1/\delta| < \rho$.
 Thus, the coordinates of points $[\alpha:\beta:\gamma:1]$ in $\check{K}$ satisfy
 $\|(\alpha,\beta,\gamma)\|<\rho$, proving that $\check{K}$ is bounded.

 Let $\Lambda=[a:b:c:1]$ and $\Lambda'=[a':b':c':1]$ be points of $\check{K}$.
 Then  $\Lambda(K),\Lambda'(K)\subset [0,\infty)$.
 Since $[0,\infty)$ is convex, for every $t\in[0,1]$, if $\Lambda_t:=t\Lambda+(1-t)\Lambda'$, then
 $\Lambda_t(K)\subset[0,\infty)$ and so $\Lambda_t\in\check{K}$.
 This proves that $\check{K}$ is convex.
\end{proof}
%%%%%%%%%%%%%%%%%%%%%%%%%%%%%%%%%%%%%%%%%%%%%%%%%%%%%%%%%%%%%%%%%%%%%%%%%%%%%%%%%

Points in the boundary ($\partial K$) of $\check{K}$ are hyperplanes supporting $K$, and faces of  $\check{K}$
correspond to exposed faces of $K$.
For example, $\check{\Pi}_i\in\partial \check{K}$ if and only if the plane $\Pi_i$ of $C_i$ supports a
two-dimensional face of $K$.
Points of the curve in $\partial \check{K}$ where the cones $\check{C}_1$ and $\check{C}_2$ meet correspond to
stationary bisecants, and line segments in the ruling of $\check{C}_i$ lying in 
$\partial\check{K}$ correspond to the exposed points of $C_i$ in $K$.
This may be seen in Fig.~\ref{Fig:duals}, 
%%%%%%%%%%%%%%%%%%%%%%%%%%%%%%%%%%%%%%%%%%%%%%%%%%%%%%%%%%%%%%%%%%%%%%%%%%%%%%%%%
\begin{figure}[htb]
  \centering
  \raisebox{-41pt}{\includegraphics[height=95pt]{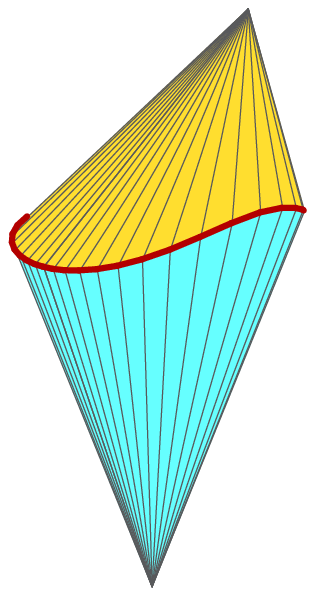}}\qquad
  \raisebox{-33pt}{\includegraphics[height=78pt]{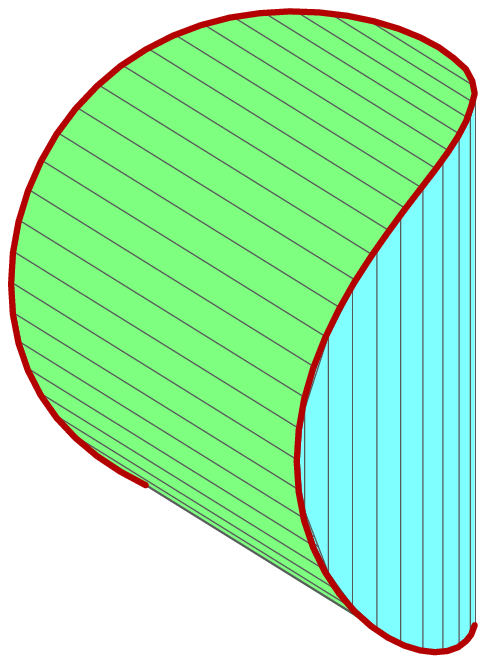}}\qquad
  \raisebox{-41pt}{\includegraphics[height=95pt]{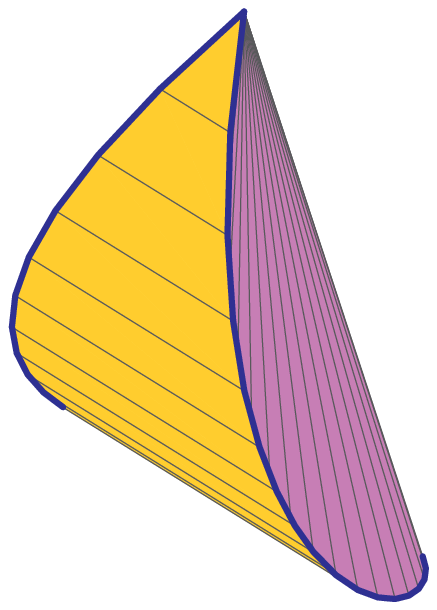}}
  \caption{Duals to convex hulls.}\label{Fig:duals}
 \end{figure}
%%%%%%%%%%%%%%%%%%%%%%%%%%%%%%%%%%%%%%%%%%%%%%%%%%%%%%%%%%%%%%%%%%%%%%%%%%%%%%%%%
which shows the dual bodies to the convex hulls of
Fig.~\ref{Fig:firstTaste}. 
For these, the origin $o$ is the midpoint of the segment joining the centers of the circles.

The intersection of two cones on the left has cone points corresponding to the planes of the 
discs in the boundary of the convex set on the left in Fig.~\ref{Fig:firstTaste}.
In the center is the dual of the oloid.
The origin $o$ is in the interior of the discs of the circles, so both cones $\check{C}_i$ are elliptical
cylinders.
On the right is the intersection of a cone with a horizontal cylinder meeting its vertex.
The cylinder is dual to the vertical circle in the rightmost convex set in Fig.~\ref{Fig:firstTaste}. 
The vertex is the 2-dimensional face, and the two branches of the intersection curve at the vertex of the cone have
limit the two nonexposed stationary bisecants.

%%%%%%%%%%%%%%%%%%%%%%%%%%%%%%%%%%%%%%%%%%%%%%%%%%%%%%%%%%%%%%%%%%%%%%%%%%%%%%%%%
%

In~\cite{ellipsoids}, the authors sketch an exact algorithm (beautifully explained in~\cite{ellipsoids_video}) to
compute the convex hull of three ellipsoids $P$, $Q$, and $R$ in $\RR^3$.
Their approach inspired the previous discussion.

If the origin $o$ lies in the interior of an ellipsoid $P$, then its dual $\check{P}$ is also an ellipsoid.
If $o$ lies on $P$, then its dual is a paraboloid and $\check{\infty}$ lies in the convex component of its complement.
If $o$ is exterior to $P$, then its dual is a hyperboloid of two sheets, and one of the convex components of its
complement contains $\check{\infty}$.

Choosing an origin $o$ in the interior of the convex hull $K$ of $P\cup Q\cup R$ as in
Proposition~\ref{P:convex-dual}, $\check{K}$ is a bounded convex set that is the closure of the region in the
complement of the duals containing the origin $\check{\infty}$. 
The video~\cite{ellipsoids_video} describes the algorithm to compute $K$ when the origin $o$ lies in the interior
of all three ellipsoids.
In that case, the dual $\check{K}$ of the convex hull of the three ellipsoids is the intersection of the three dual
ellipsoids $\check{P}\cap\check{Q}\cap\check{R}$.
Computing $\check{K}$ requires the computation of the curves where two dual ellipsoids intersect,
and points where three dual ellipsoids meet, and then decomposing the dual ellipsoids along these curves
into patches.

This analysis gives three types of points in the boundary of $\check{K}$.
\begin{enumerate}
 \item Points common to all three dual ellipsoids.
       These give tritangent planes in $\partial K$.

 \item Points on curves given by the pairwise intersection of dual ellipsoids.
       They are bitangent planes and give bitangent edges.
       These form 1-dimensional families of 1-faces in $\partial K$.

 \item Points on a single dual ellipsoid.  
       These are tangent planes to an ellipsoid at a point of $K$, and give a two-dimensional family of exposed
       points of $K$ coming from the corresponding ellipsoid.
\end{enumerate}

As we see in Fig.~\ref{Fig:duals}, the dual $\check{K}$ eloquently displays information about the exposed
faces of $K$, but information about the nonexposed faces is less clear in $\check{K}$.

%%%%%%%%%%%%%%%%%%%%%%%%%%%%%%%%%%%%%%%%%%%%%%%%%%%%%%%%%%%%%%%%%%%%%%%%%%%%%%%%%
\providecommand{\bysame}{\leavevmode\hbox to3em{\hrulefill}\thinspace}
\providecommand{\MR}{\relax\ifhmode\unskip\space\fi MR }
% \MRhref is called by the amsart/book/proc definition of \MR.
\providecommand{\MRhref}[2]{%
  \href{http://www.ams.org/mathscinet-getitem?mr=#1}{#2}
}
\providecommand{\href}[2]{#2}

%%%%%%%%%%%%%%%%%%%%%%%%%%%%%%%%%%%%%%%%%%%%%%%%%%%%%%%%%%%%%%%%%%%%%%%%%%%%%%%%%

\begin{thebibliography}{10}

\bibitem{barvinok}
Alexander Barvinok, \emph{A course in convexity}, Graduate Studies in
  Mathematics, vol.~54, American Mathematical Society, Providence, RI, 2002.

\bibitem{Oloid_AG}
Inversions-Technik~GmbH Basle, {\tt http://www.oloid.ch/index.php/en/},
  formerly Oloid AG.

\bibitem{Oloid_developable}
Hans Dirnb{\"o}ck and Hellmuth Stachel, \emph{The development of the oloid}, J.
  Geom. Graph. \textbf{1} (1997), no.~2, 105--118.

\bibitem{multidegree}
Laura Escobar and Allen Knutson, \emph{Secants, bitangents, and their
  congruences}, in {\it Combinatorial Algebraic Geometry} (eds. G.G. Smith and
  B. Sturmfels), to appear.

\bibitem{Finch}
Steven~R. Finch, \emph{Convex hull of two orthogonal disks}, {\tt
  arxiv.org/1211.4514}, 2012.

\bibitem{ellipsoids}
Nicola Geismann, Michael Hemmer, and Elmar Sch\"omer, \emph{The convex hull of
  ellipsoids}, SCG '01 Proceedings of the seventeenth annual symposium on
  Computational geometry, ACM, New York, 2001, pp.~321--322.

\bibitem{ellipsoids_video}
\bysame, \emph{The convex hull of ellipsoids 2001}, {\tt
  https://youtu.be/Tq9OS5iIcBc}, 2001.

\bibitem{Gr03}
Branko Gr{\"u}nbaum, \emph{Convex polytopes}, second ed., Graduate Texts in
  Mathematics, vol. 221, Springer-Verlag, New York, 2003, Prepared and with a
  preface by Volker Kaibel, Victor Klee and G{\"u}nter M. Ziegler.

\bibitem{hartshorne}
Robin Hartshorne, \emph{Algebraic geometry}, Springer-Verlag, New
  York-Heidelberg, 1977, Graduate Texts in Mathematics, No. 52.

\bibitem{TCR}
Hiroshi Ira, \emph{The development of the two-circle-roller in a numerical
  way}, unpublished note, {\tt http://ilabo.bufsiz.jp/}, 2011.

\bibitem{Johnsen}
Trygve Johnsen, \emph{Plane projections of a smooth space curve}, Parameter
  spaces ({W}arsaw, 1994), Banach Center Publ., vol.~36, Polish Acad. Sci.,
  Warsaw, 1996, pp.~89--110.

\bibitem{congruences}
Kathl\'en Kohn, Bernt Ivar~Utst{\o}l N{\o}dland, and Paolo Tripoli, \emph{The
  multidegree of the multi-image variety}, in {\it Combinatorial Algebraic
  Geometry} (eds. G.G. Smith and B. Sturmfels), to appear.

\bibitem{RS11}
Kristian Ranestad and Bernd Sturmfels, \emph{The convex hull of a variety},
  Notions of Positivity and the Geometry of Polynomials (Petter Br\"and\'en,
  Mikael Passare, and Mihai Putinar, eds.), Springer, 2011, pp.~331--344.

\bibitem{RS12}
\bysame, \emph{On the convex hull of a space curve}, Advances in Geometry
  \textbf{12} (2012), 157--178.

\bibitem{orbitopes}
Raman Sanyal, Frank Sottile, and Bernd Sturmfels, \emph{Orbitopes}, Mathematika
  \textbf{57} (2011), no.~2, 275--314.

\bibitem{Schatz}
Paul Schatz, \emph{Oloid, a device to generate a tumbling motion}, Swiss Patent
  No.~500,000, 1929.

\bibitem{Sch12}
Claus Scheiderer, \emph{Semidefinite representation for convex hulls of real
  algebraic curves}, {\tt arXiv.org/1208.3865}, 2012.

\bibitem{Sch16}
\bysame, \emph{Semidefinitely representable convex sets}, {\tt
  arXiv.org/1617.07048}, 2016.

\bibitem{Seidenberg}
Abraham Seidenberg, \emph{A new decision method for elementary algebra}, Ann.
  of Math. (2) \textbf{60} (1954), 365--374.

\bibitem{Sinn}
Rainer Sinn, \emph{Algebraic boundaries of {${\rm SO}(2)$}-orbitopes}, Discrete
  Comput. Geom. \textbf{50} (2013), no.~1, 219--235.

\bibitem{Sturmfels}
Bernd Sturmfels, \emph{Fitness, apprenticeship, and polynomials}, in
  \emph{Combinatorial Algebraic Geometry}, eds. G.G.Smith and B.Sturmfels, to
  appear.

\bibitem{Tarski}
Alfred Tarski, \emph{A {D}ecision {M}ethod for {E}lementary {A}lgebra and
  {G}eometry}, RAND Corporation, Santa Monica, Calif., 1948.

\bibitem{Vinzant}
Cynthia Vinzant, \emph{Edges of the {B}arvinok-{N}ovik orbitope}, Discrete
  Comput. Geom. \textbf{46} (2011), no.~3, 479--487.

\bibitem{Z95}
G{\"u}nter~M. Ziegler, \emph{Lectures on polytopes}, Graduate Texts in
  Mathematics, vol. 152, Springer-Verlag, New York, 1995.

\end{thebibliography}
\end{document}